\documentclass[a4paper]{amsart}
\usepackage{latexsym,bm,stmaryrd}
\usepackage{amsmath,amsthm,amsfonts,amssymb,mathrsfs}

\usepackage{times,amssymb,amsmath,mathrsfs}

\let\<=\langle
\let\>=\rangle
\def\({\big(}
\def\){\big)}

\let\Sym=\BS

\newcommand{\N}{\mathbb N}
\newcommand{\Z}{\mathbb Z}

\newcommand{\HH}{\mathscr{H}}

\DeclareMathAlphabet{\mathpzc}{OT1}{pzc}{m}{it}

\newcommand{\Q}{\mathbb Q}

\def\s{{\mathfrak s}}
\def\t{{\mathfrak t}}

\newcommand{\eps}{\varepsilon}

\renewcommand\t{\mathfrak{t}}

\DeclareMathOperator{\TPro}{TProd}

\DeclareMathOperator{\Pro}{Prod}

\newcounter{main}
\theoremstyle{plain}

\numberwithin{equation}{section}
\newtheorem{prop}[equation]{Proposition}
\newtheorem{thm}[equation]{Theorem}
\newtheorem{cor}[equation]{Corollary}
\newtheorem{lem}[equation]{Lemma}

\theoremstyle{definition}
\newtheorem{dfn}[equation]{Definition}
\theoremstyle{remark}
\newtheorem{rem}[equation]{Remark}

  {\refstepcounter{equation}\trivlist
   \item[\hskip\labelsep\theequation.~\textbf{Example#1}\space]
   \ignorespaces
  }{\unskip\nobreak\hfil%
    \penalty50\hskip2em\hbox{}\nobreak\hfil$\Diamond$%
    \parfillskip=0pt\finalhyphendemerits=0\penalty-100\endtrivlist}

{\catcode`\|=\active
  \gdef\set#1{\mathinner{\lbrace\,{\mathcode`\|"8000%
                                   \let|\midvert #1}\,\rbrace}}
}
\def\midvert{\egroup\mid\bgroup}

\begin{document}
\title[Involutions in Weyl groups]{On involutions in Weyl groups}

\author[J. Hu]{Jun Hu}
  \address{Department of Mathematics\\
  Zhejiang University\\
  Hangzhou, 310027, P.R. China}
  \email{junhu303@qq.com}
  
\author[J. Zhang]{Jing Zhang}
   \address{School of Mathematics and Statistics\\
  Beijing Institute of Technology\\
  Beijing, 100081, P.R. China}
\email{ellenbox@bit.edu.cn}
  
\begin{abstract}
Let $(W,S)$ be a Coxeter system and $\ast$ be an automorphism of $W$ with order $\leq 2$ such that $s^{\ast}\in S$ for any $s\in S$.
Let $I_{\ast}$ be the set of twisted involutions relative to $\ast$ in $W$. In this paper we consider the case when $\ast=\text{id}$ and study the braid
$I_\ast$-transformations between the reduced $I_\ast$-expressions of involutions. If $W$ is the Weyl group of type $B_n$ or $D_n$, we explicitly describe a
finite set of basic braid $I_\ast$-transformations for all $n$ simultaneously, and show that any two reduced $I_\ast$-expressions for a given involution can
be transformed into each other through a series of basic braid $I_\ast$-transformations. In both cases, these basic braid $I_\ast$-transformations consist of
the usual basic braid transformations plus some natural ``right end transformations" and plus exactly one extra transformation. The main result generalizes our
previous work for the Weyl group of type $A_{n}$.
\end{abstract}

\maketitle


\section{Introduction}

Let $(W,S)$ be a fixed Coxeter system with length function $\ell: W\rightarrow\mathbb{N}$. If $w\in W$ then by definition
$$\ell(w):=\min\{k|w=s_{i_1}\dots s_{i_k} \text{ for some }s_{i_1},\dots,s_{i_k}\in S\}.$$
Let $\leq$ be the Bruhat partial ordering on $W$ defined with respect to $S$. Let $\ast$ be a fixed automorphism of $W$ with order $\leq 2$ and such that $s^{\ast}\in S$ for any $s\in S$.

\begin{dfn} \label{twistedinvolutions} We define $I_{\ast}:=\bigl\{w\in W\bigm|w^{\ast}=w^{-1}\bigr\}$.
The elements of $I_{\ast}$ will be called twisted involutions relative to $\ast$.
\end{dfn}

If $\ast=\text{id}_W$ (the identity automorphism on $W$), then the elements of $I_{\ast}$ will be called involutions.

\begin{dfn} \label{twistedinvolutions2} For any $w\in I_\ast$ and $s\in S$, we define $$
s\ltimes w:=\begin{cases} sw &\text{if $sw=ws^\ast$;}\\
sws^\ast &\text{if $sw\neq ws^\ast$.}
\end{cases}
$$
For any $w\in I_\ast$ and $s_{i_1},\cdots,s_{i_k}\in S$, we define $$
s_{i_1}\ltimes s_{i_2}\ltimes\cdots\ltimes s_{i_k}\ltimes w:=s_{i_1}\ltimes\bigl(s_{i_2}\ltimes\cdots\ltimes (s_{i_k}\ltimes w)\cdots\bigr) .
$$
\end{dfn}
It is clear that $s\ltimes w\in I_\ast$ whenever $w\in I_\ast$ and $s\in S$. Note that $\ltimes$ is not associative. So the above convention for how to interpret expressions without parentheses is nontrivial and meaningful.

\begin{dfn} \label{Ireduced} (\cite{RS1,Hu2,Mar,HMP2}) Let $w\in I_\ast$. If $w=s_{i_1}\ltimes s_{i_2}\ltimes\cdots\ltimes s_{i_k}$, where $k\in\N$, $s_{i_j}\in S$ for each $j$, then $(s_{i_1},\cdots,s_{i_k})$ is called an $I_\ast$-expression for $w$. Such an $I_\ast$-expression for $w$ is reduced if its length $k$ is minimal.
\end{dfn}
We regard the empty sequence $()$ as a reduced $I_\ast$-expression for $w=1$. It follows by induction on $\ell(w)$ that every element $w\in I_\ast$ has a reduced $I_\ast$-expression.

The phrase ``$I_\ast$-expression" dates originally to Marberg \cite{Mar}. Hultman \cite{Hu2} uses the right-handed version of $I_\ast$-expression and call it $\underline{S}$-expression. In recent papers, Hamaker, Marberg and Pawlowski have started calling reduced $\underline{S}$-expressions ``involution words". Going much further back, reduced $\underline{S}$-expressions and involution words are both the same as what Richardson and Springer called ``admissible sequences" in \cite[Section 3]{RS1}. Our notation in the current paper follows the conventions from \cite{Mar}.

A well-known classical fact of Matsumoto (\cite{Mat}) says that any two reduced expressions for an element in $W$ can be transformed into each other through a series of basic braid transformations. We are interested in finding the right analogue of this fact for twisted involution relative to $\ast$ with respect to the operation ``$\ltimes$". In this paper we consider the case when $\ast=\text{id}$. If $W$ is the Weyl group of type $B_n$ or type $D_n$, we identify a set of basic braid $I_\ast$-transformations which span and preserve the sets of reduced $I_\ast$-expressions for any involution. These results generalize our earlier work in \cite{HZH1} for the Weyl group of type $A_{n}$. Note that these generalizations are non-trivial in the sense that the basic braid $I_\ast$-transformations for the Weyl group of types $B_n$ and $D_n$ which we identify contain not only the usual basic braid transformations plus some natural ``right end transformations" but also one extra transformations which do not directly related to the usual basic braid transformations; see the last relation in Definition \ref{braid1} and \ref{braid1B}. This is a new phenomenon which does not happen in type $A$ case; compare \cite[Definition 2.12]{HZH1}.

There have been a number of important works on algebraic and combinatorial properties of involutions in Coxeter groups. They have arisen independently in a few different geometric contexts; see \cite{RS1, RS2, Hu1, Hu2, CJ, CJW, HMP1, HMP2, WY}. Richardson and Springer \cite{RS1} first initiates the study of Bruhat order restricted to involutions in finite Weyl group, which naturally leads to the consideration of $I_\ast$-expression (which they called ``admissible sequence"). They proved (\cite[Lemma 3.16]{RS1}) that the set of reduced $I_\ast$-expression for a given twisted involution is closed under the ordinary braid relations for $(W,S)$; see also \cite[Proposition 1.4]{HMP2} for an equivalent statement for general Coxeter groups. Can, Joyce and Wyser have classified these sets for each involution in the symmetric group. It is a natural (and nontrivial) open problem to use the main result of the current paper to extend Can, Joyce and Wyser's results to types $B$ and $D$. The paper \cite{HMP2} of Hamaker, Marberg and Pawlowski also proves an analogue of Matsumoto's theorem for the right handed versions of reduced $I_\ast$-expressions. Their result applies to arbitrary Coxeter groups, but requires a potentially unbounded number of extra relations in addition to the ordinary braid relations and the ``right end transformation" described in Remark \ref{rem0} and \ref{rem0B}. We thank the anonymous referee for the detailed explanation of these facts to us. The main result in this paper show that in types $B$ and $D$ one can ignore all but one of these extra relations.

Our motivation for the study of reduced $I_\ast$-expressions for involution comes from a conjecture of Lusztig.  Let $v$ be an indeterminate over $\Z$ and $u:=v^2$. Set $A:=\Z[u^2,u^{-2}]$, $\mathcal{A}:=\Z[u,u^{-1}]$. Let $\HH_{u^2}$ be the one-parameter Iwahori--Hecke algebra associated to $(W,S)$ with Hecke parameter $u^2$ and defined over $A$ (cf. \cite{Hum}). Let $\mathcal{H}_u:=\mathcal{A}\otimes_{A}\HH_{u^2}$. We abbreviate $1_{\mathcal{A}}\otimes _A T_w$ as $T_w$ for each $w\in W$. Let $M$ be the free $\mathcal{A}$-module with basis $\{a_{w}|w\in I_{\ast}\}$. An $\mathcal{H}_u$-module structure on $M$ was introduced by Lusztig and Vogan (\cite{LV1}) in the special case when $W$ is a Weyl group or an affine Weyl group, and by Lusztig (\cite{Lu1}) in the general case. When $u$ is specialized to $1$, the module $M$ was introduced more than fifteen years ago by Kottwitz. Kottwitz found the module by analyzing Langlands' theory of stable characters for real groups. He gave a conjectural description of it (later established by Casselman) in terms of the Kazhdan-Lusztig left
cell representations of W. For these reasons it was clear that $M$ was an interesting, subtle, and important object. In \cite[3.4(a)]{Lu2} Lusztig defined $X_{\emptyset}:=\sum_{x\in W, x^\ast=x}u^{-\ell(x)}T_x\in\mathcal{H}_u$ and he conjectured that there is a unique isomorphism of $(\Q(u)\otimes_{\mathcal{A}}\mathcal{H}_u)$-modules
$\eta: \Q(u)\otimes_{\mathcal{A}}M\cong (\Q(u)\otimes_{\mathcal{A}}\mathcal{H}_u)X_{\emptyset}$ such that $\eta(a_1)=X_{\emptyset}$.

In \cite{HZH1}, we give a proof of this conjecture when $\ast=\text{id}_W$ and $W$ is the Weyl group of type $A_{n}$ for any $n\in\N$. The key ingredient in the proof is to prove an analogue of Matsumoto's result for reduced $I_\ast$-expressions of involutions. We announced in that paper that the same strategy should work for the Weyl groups of other types. Later Lusztig proved his conjecture (in \cite{Lu3}) for any Coxeter group and any $\ast$ by using a completely different argument. Despite this fact, it is still interesting in itself to generalize Matsumoto's result for reduced $I_\ast$-expressions of involutions to Weyl groups of arbitrary types (other than type $A$). In this paper we give this generalization for the Weyl groups of types $B_n$ and $D_n$ by finding a finite set of basic braid $I_\ast$-transformations on reduced $I_{\ast}$-expressions for involutions, which can be described for all $n$ simultaneously and whose size depends quadratically on $n$.

The paper is organised as follows. In Section 2 we recall some preliminary results on reduced $I_{\ast}$-expressions for twisted involutions relative $\ast$. Based on the work of Lusztig \cite[1.2, 1.4]{Lu1}, we give a case-by-case discussion after Lemma \ref{7cases} when $\ast=\text{id}$ which will be used in the next two sections. In Section 3 we consider the Weyl group $W(D_n)$ of type $D_n$ and give the definition of basic braid $I_\ast$-transformation on reduced $I_{\ast}$-expressions for involutions in $W(D_n)$ in Definition \ref{braid1}. In Section 4 we consider the Weyl group $W(B_n)$ of type $B_n$ and give the definition of basic braid $I_\ast$-transformation on reduced $I_{\ast}$-expressions for involutions in $W(B_n)$ in Definition \ref{braid1B}. The main results are Theorem \ref{mainthm0} and  \ref{mainthm0B}, where we show that any two reduced $I_\ast$-expressions for an involution in $W\in\{W(D_n),W(B_n)\}$ can be transformed into each other through a series of braid $I_\ast$-transformations. In Section 5 we use the main result in Section 3 and Section 4 to show that $\eta$ is a well-defined surjective $(\Q(u)\otimes_{\mathcal{A}}\mathcal{H}_u)$-module homomorphism when $W$ is the Weyl group of type $B_n$ or $D_n$ and $\ast=\text{id}$.
\medskip

\section*{Acknowledgements}

Both authors were supported by the National Natural Science Foundation of China (NSFC 11525102).  The first author was also supported by the research fund
from Zhejiang University. Both authors are grateful to the anonymous referee for his/her careful reading and very helpful comments and suggestions.

\bigskip
\section{reduced $I_{\ast}$-expressions}

Let $(W,S)$ be a fixed Coxeter system with length function $\ell: W\rightarrow\mathbb{N}$. For any $s\neq t\in S$, let $m_{s,t}=m_{t,s}\in[2,\infty]$ be the order of $st$. The following facts can all be found in \cite{Hu1,Hu2} and \cite{HZH1}.

\begin{lem} \label{square} \text{(\cite{Hu1})} For any $w\in I_\ast$ and $s\in S$, we have that $$
s\ltimes(s\ltimes w)=w .
$$
\end{lem}

It is well-known that every element $w\in I_\ast$ is of the form $w=s_{i_1}\ltimes s_{i_2}\ltimes\cdots\ltimes s_{i_k}$ for some $k\in\N$ and $s_{i_1},\cdots,s_{i_k}\in S$.

\begin{lem}\label{rankfunc}(\cite{Hu1}, \cite{Hu2}) Let $w\in I_\ast$. Any reduced $I_\ast$-expression for $w$ has a common length. Let $\rho: I_\ast\rightarrow\N$ be the map which assigns $w\in I_\ast$ to this common length. Then $(I_\ast,\leq)$ is a graded poset with rank function $\rho$. Moreover, if $s\in S$ then $\rho(s\ltimes w)=\rho(w)\pm 1$, and $\rho(s\ltimes w)=\rho(w)-1$ if and only if $\ell(sw)=\ell(w)-1$.
\end{lem}

\begin{cor} \label{length0} \text{(\cite[Corollary 2.6]{HZH1})} Let $w\in I_\ast$ and $s\in S$. Suppose that $sw\neq ws^\ast$. Then $\ell(sw)=\ell(w)+1$ if and only if $\ell(ws^\ast)=\ell(w)+1$, and if and only if $\ell(s\ltimes w)=\ell(w)+2$. The same is true if we replace ``$+$" by ``$-$".
\end{cor}

\begin{cor} \label{deletion2} \text{(\cite[Corollary 2.7]{HZH1})} Let $w\in I_\ast$ and $s\in S$. Suppose that $\rho(w)=k$. If $sw<w$ then $w$ has a reduced $I_\ast$-expression which is of the form $s\ltimes s_{j_1}\ltimes\cdots\ltimes s_{j_{k-1}}$.
\end{cor}

\begin{dfn} \label{ReducedSequence} \text{(\cite[Definition 2.8]{HZH1})} Let $w\in I_\ast$ and $s_{i_1},\cdots,s_{i_k}\in S$. If $$
\rho(s_{i_1}\ltimes s_{i_2}\ltimes\cdots\ltimes s_{i_k}\ltimes w)=\rho(w)+k ,
$$
then we shall call the sequence $(s_{i_1},\cdots,s_{i_k},w)$ reduced, or $(s_{i_1},\cdots,s_{i_k},w)$  a reduced sequence.
\end{dfn}

In particular, any reduced $I_\ast$-expression for $w\in I_\ast$ is automatically a reduced sequence.
In the sequel, by some abuse of notations, we shall also call $(i_1,\cdots,i_k)$ a reduced sequence whenever $(s_{i_1},\cdots,s_{i_k})$ is a reduced sequence in the sense of Definition \ref{ReducedSequence}.

\begin{rem} Let $s_{i_1},\cdots,s_{i_k}\in S$ and $1\leq a\leq k$. We shall use the expression \begin{equation}\label{omit}
s_{i_1}\ltimes\cdots\ltimes s_{i_{a-1}}\ltimes s_{i_{a+1}}\ltimes\cdots\ltimes s_{i_k}
\end{equation}
to denote the element obtained by omitting ``$s_{i_a}\ltimes$" in the expression $s_{i_1}\ltimes\cdots\ltimes s_{i_k}$. In particular, if $a=1$ then (\ref{omit}) denotes the element $s_{i_2}\ltimes\cdots\ltimes s_{i_k}$; while if $a=k$ then (\ref{omit}) denotes the element $s_{i_1}\ltimes\cdots\ltimes s_{i_{k-1}}$.
\end{rem}

\begin{prop} \text{(Exchange Property,\,\, \cite[Prop. 3.10]{Hu2})} \label{exchange} Suppose $(s_{i_1}, \cdots,s_{i_k})$ is a reduced $I_\ast$-expression for $w\in I_\ast$ and that $\rho(s\ltimes s_{i_1}\ltimes s_{i_2}\ltimes\cdots\ltimes s_{i_k})<k$ for some $s\in S$. Then$$
s\ltimes s_{i_1}\ltimes s_{i_2}\ltimes\cdots\ltimes s_{i_k}=s_{i_1}\ltimes s_{i_2}\ltimes\cdots\ltimes{s_{i_{a-1}}}\ltimes{s_{i_{a+1}}}\ltimes\cdots\ltimes s_{i_k}
$$
for some $a\in\{1,2,\cdots,k\}$.
\end{prop}

For any $s,t\in S$, recall that $m_{st}$ is the order of $st$. Then $m_{st}$ is the order of $ts$ too. Suppose that $m_{st}<\infty$. We define $$
\Pro(s,t; m_{st}):=\underbrace{stst\cdots}_{\text{$m_{st}$ factors}} .
$$
By definition, we have that $\Pro(s,t; m_{st})=\Pro(t,s; m_{st})$. In this case, we shall call the transformation $\Pro(s,t; m_{st})\longleftrightarrow\Pro(t,s; m_{st})$ the {\it (usual) basic braid transformation}. By a {\it (usual) braid transformation} on a given reduced expression, we mean the compositions of a series of (usual) basic braid transformations.

For any $s,t\in S$ with  $m_{st}<\infty$, we define $$
\TPro(s,t; m_{st}):=\underbrace{s\ltimes t\ltimes s\ltimes t\ltimes \cdots }_{\text{$m_{st}$ factors}}\ltimes .
$$
Note that $\Pro(s,t; m_{st})$ is an element of $W$, while $\TPro(s,t; m_{st})$ is an operator on $I_\ast$ instead of an element of $W$.

Our purpose is to find the right analogues of (basic) braid transformations for twisted involutions in $I_\ast$ and the operation ``$\ltimes$". The following result amounts to saying that any usual (basic) braid transformation naturally induces a (basic) braid $I_\ast$-transformation on reduced $I_{\ast}$-expression.

\begin{prop} \label{braid} \text{(\cite[Lemma 3.16]{RS1}, \cite[Proposition 1.4]{HMP2})} Let $1\neq w\in I_\ast$ and $s,t\in S$. Suppose that $2\leq m_{st}<\infty$, and $(\underbrace{s,t,s,t,\dots}_{\text{$m_{st}$ factors}}, w)$ is a reduced $I_{\ast}$-expression. Then $$\TPro(s,t;m_{st}) w = \TPro(t,s;m_{st}) w. $$
\end{prop}

\textbf{In the rest of this paper, we assume that $\ast=\text{id}$.} In particular, $I_{\ast}=\{w\in W|w^2=1\}$ is the set of involutions in $W$. For any $s\in S$, it holds that $s\ltimes 1=s$.

If $w\in I_\ast$, then a simple reflection $s\in S$ is called a descent of $w$ whenever $sw<w$ (equivalently, $ws<w$).

\begin{lem} \label{7cases} \text{(\cite[Section 2.1]{GP})} Suppose that $w\in I_\ast$. Let $s\ltimes s_{i_1}\ltimes\dots\ltimes s_{i_k}$ and $t\ltimes s_{j_1}\ltimes\dots\ltimes s_{j_k}$ be two reduced $I_\ast$-expressions of $w$ such that $s\neq t$. Let $K:=\{s,t\}$. Let $W_K$ be the subgroup of $W$ generated by $K$ and $\Omega:=W_K wW_K$. Suppose that $|W_K|<\infty$. Let $b\in W$ be the unique minimal length $(W_K,W_K)$-double coset representative in $\Omega$. Then $b\in I_{\ast}$ and $w$ is the unique maximal length $(W_K,W_K)$-double coset representative in $\Omega$. Moreover, $w=w_{0,K}bw_{0,J}w_{0,K}=w_{0,K}w_{0,J}bw_{0,K}$, where $J:=K\cap bKb$, $w_{0,J}$ and $w_{0,K}$ are the unique longest elements in $W_J$ and $W_K$ respectively.
\end{lem}

\begin{proof} Since $b\in W$ is the unique minimal length $(W_K,W_K)$-double coset representative in $\Omega$, it follows that
$b^{-1}\in W$ is also the unique minimal length $(W_K,W_K)$-double coset representative in $\Omega$. Hence $b=b^{-1}$ and $b\in I_\ast$. This proves the first statement of the lemma. Applying Lemma \ref{rankfunc}, we see that $sw<w>tw$. Since $w=w^{-1}$, it follows that $ws<w^{-1}=w>wt$. Now the second statement of the lemma follows from \cite[Corollary 4.19]{DDPW}.
\end{proof}

Let $w\in I_\ast$, and suppose $s\neq t$ in $S$ are such that $\ell(sw)=\ell(tw)<\ell(w)$. Define $\Omega:=W_K wW_K$. Suppose that $|W_K|<\infty$. Let $b\in W$ be the unique minimal length element in $\Omega$. Then $b\in I_\ast$ by Lemma \ref{7cases}. Let $m:=m_{s,t}=m_{t,s}$ be the order of $st$. Following \cite{Lu1}, for each $1\leq i\leq m$, we set $$
\s_i:=\underbrace{sts\dots}_{\text{$i$ factors}},\quad \t_i:=\underbrace{tst\dots}_{\text{$i$ factors}}.
$$
Set $J:=K\cap bKb^{-1}$. By \cite[1.2(a), 1.4]{Lu1}, there are only the following seven cases. \medskip

{\it Case 1.} $\{sb, tb\}\cap\{bs, bt\}=\emptyset$, $J=\emptyset$, $\Omega\cap I_{\ast}=\{\xi_{2i},\xi'_{2i}|0\leq i\leq m\}$, where $\xi_0=\xi'_0=b$, $\xi_{2m}=\xi'_{2m}=w$, $\xi_{2i}=\s_i^{-1}b\s_i$,  $\xi'_{2i}=\t_i^{-1}b\t_i$, $\ell(\xi_{2i})=\ell(\xi'_{2i})=\ell(b)+2i$;
In this case, $$\begin{aligned}w=w_{0,K}bw_{0,K}&=\underbrace{sts\dots}_{\text{$m$ factors}}b\underbrace{sts\dots}_{\text{$m$ factors}}=
\underbrace{tst\dots}_{\text{$m$ factors}}b\underbrace{tst\dots}_{\text{$m$ factors}} \\
&=\underbrace{sts\dots}_{\text{$m$ factors}}b\underbrace{tst\dots}_{\text{$m$ factors}}=
\underbrace{tst\dots}_{\text{$m$ factors}}b\underbrace{sts\dots}_{\text{$m$ factors}}. \end{aligned} $$
By length consideration we can deduce that $$
w=\underbrace{s\ltimes t\ltimes s\ltimes\dots}_{\text{$m$ factors}}\ltimes b=\underbrace{t\ltimes s\ltimes t\ltimes\dots}_{\text{$m$ factors}}\ltimes b .
$$

{\it Case 2.} $sb=bs, tb\neq bt$, $J=\{s\}$, $\Omega\cap I_{\ast}=\{\xi_{2i},\xi_{2i+1}|0\leq i\leq m-1\}$, where $\xi_0=b, \xi_{2m-1}=w$, $\xi_{2i}=\t_i^{-1}b\t_i$,  $\xi_{2i+1}=\t_i^{-1}b\s_{i+1}=\s_{i+1}^{-1}b\t_i$, $\ell(\xi_{2i})=\ell(b)+2i$, $\ell(\xi_{2i+1})=\ell(b)+2i+1$;
In this case, $$\begin{aligned}w=w_{0,K}w_{0,J}bw_{0,K}&=\underbrace{sts\dots}_{\text{$m$ factors}}(sb)\underbrace{sts\dots}_{\text{$m$ factors}}=
\underbrace{sts\dots}_{\text{$m$ factors}}(bs)\underbrace{sts\dots}_{\text{$m$ factors}} \\
&=\underbrace{sts\dots}_{\text{$m$ factors}}b\underbrace{tst\dots}_{\text{$m-1$ factors}}=
\underbrace{tst\dots}_{\text{$m$ factors}}b\underbrace{tst\dots}_{\text{$m-1$ factors}}. \end{aligned} $$
By length consideration we can deduce that $$
w=\underbrace{s\ltimes t\ltimes s\ltimes\dots}_{\text{$m$ factors}}\ltimes b=\underbrace{t\ltimes s\ltimes t\ltimes\dots}_{\text{$m$ factors}}\ltimes b .
$$

\medskip
{\it Case 3.} $sb\neq bs, tb=bt$, $J=\{t\}$, $\Omega\cap I_{\ast}=\{\xi_{2i},\xi_{2i+1}|0\leq i\leq m-1\}$, where $\xi_0=b, \xi_{2m-1}=w$, $\xi_{2i}=\s_i^{-1}b\s_i$,  $\xi_{2i+1}=\s_i^{-1}b\t_{i+1}=\t_{i+1}^{-1}b\s_i$, $\ell(\xi_{2i})=\ell(b)+2i$, $\ell(\xi_{2i+1})=\ell(b)+2i+1$;
In this case, $$\begin{aligned}w=w_{0,K}w_{0,J}bw_{0,K}&=\underbrace{sts\dots}_{\text{$m$ factors}}(tb)\underbrace{sts\dots}_{\text{$m$ factors}}=
\underbrace{sts\dots}_{\text{$m$ factors}}(bt)\underbrace{tst\dots}_{\text{$m$ factors}} \\
&=\underbrace{sts\dots}_{\text{$m$ factors}}b\underbrace{sts\dots}_{\text{$m-1$ factors}}=
\underbrace{tst\dots}_{\text{$m$ factors}}b\underbrace{sts\dots}_{\text{$m-1$ factors}}. \end{aligned} $$
By length consideration we can deduce that $$
w=\underbrace{s\ltimes t\ltimes s\ltimes\dots}_{\text{$m$ factors}}\ltimes b=\underbrace{t\ltimes s\ltimes t\ltimes\dots}_{\text{$m$ factors}}\ltimes b .
$$

\medskip
{\it Case 4.} $sb=bs, tb=bt$, $J=K$, $m$ is odd, $\Omega\cap I_{\ast}=\{\xi_0=\xi'_0=b, \xi_m=\xi'_m=w, \xi_{2i+1},\xi'_{2i+1}|0\leq i\leq (m-1)/2\}$, where
$\xi_1=sb$, $\xi_3=tstb$, $\xi_5=ststsb$, $\dots$; $\xi'_1=tb$, $\xi'_3=stsb$, $\xi'_5=tststb$, $\dots$; $\ell(\xi_{2i+1})=\ell(\xi'_{2i+1})=\ell(b)+2i+1$;
In this case, $$w=w_{0,K}w_{0,J}bw_{0,K}=bw_{0,K}=\underbrace{sts\dots s}_{\text{$m$ factors}}(b)=
\underbrace{tst\dots t}_{\text{$m$ factors}}(b).  $$
By length consideration we can deduce that $$
w=\underbrace{s\ltimes t\ltimes s\ltimes\dots}_{\text{$(m+1)/2$ factors}}\ltimes b=\underbrace{t\ltimes s\ltimes t\ltimes\dots}_{\text{$(m+1)/2$ factors}}\ltimes b .
$$

\medskip
{\it Case 5.} $sb=bs, tb=bt$, $J=K$, $m$ is even, $\Omega\cap I_{\ast}=\{\xi_0=\xi'_0=b, \xi_m=\xi'_m=w, \xi_{2i+1},\xi'_{2i+1}|0\leq i\leq (m-2)/2\}$, where
$\xi_1=sb$, $\xi_3=tstb$, $\xi_5=ststsb$, $\dots$; $\xi'_1=tb$, $\xi'_3=stsb$, $\xi'_5=tststb$, $\dots$; $\ell(\xi_{2i+1})=\ell(\xi'_{2i+1})=\ell(b)+2i+1$;
$\xi_m=\xi'_m=b\s_m=b\t_m=\s_mb=\t_mb$; $\ell(\xi_{m})=\ell(\xi'_{m})=\ell(b)+m$.
In this case, using a similar argument as in Case 4, we can get that $$
w=\underbrace{s\ltimes t\ltimes s\ltimes\dots}_{\text{$m/2+1$ factors}}\ltimes b=\underbrace{t\ltimes s\ltimes t\ltimes\dots}_{\text{$m/2+1$ factors}}\ltimes b .
$$

\medskip
{\it Case 6.} $sb=bt, tb=bs$, $J=K$, $m$ is odd, $\Omega\cap I_{\ast}=\{\xi_0=\xi'_0=b, \xi_m=\xi'_m=w, \xi_{2i},\xi'_{2i}|0\leq i\leq (m-1)/2\}$, where
$\xi_2=stb$, $\xi_4=tstsb$, $\xi_6=stststb$, $\dots$; $\xi'_2=tsb$, $\xi'_4=ststb$, $\xi'_6=tststsb$, $\dots$; $\ell(\xi_{2i})=\ell(\xi'_{2i})=\ell(b)+2i$;
$\xi_m=\xi'_m=b\s_m=b\t_m=\s_mb=\t_mb$; $\ell(\xi_{m})=\ell(\xi'_{m})=\ell(b)+m$.
In this case, using a similar argument as in Case 4, we can get that $$
w=\underbrace{s\ltimes t\ltimes s\ltimes\dots}_{\text{$(m+1)/2$ factors}}\ltimes b=\underbrace{t\ltimes s\ltimes t\ltimes\dots}_{\text{$(m+1)/2$ factors}}\ltimes b .
$$

\medskip
{\it Case 7.} $sb=bt, tb=bs$, $J=K$, $m$ is even, $\Omega\cap I_{\ast}=\{\xi_0=\xi'_0=b, \xi_m=\xi'_m=w, \xi_{2i},\xi'_{2i}|0\leq i\leq m/2\}$, where
$\xi_2=stb$, $\xi_4=tstsb$, $\xi_6=stststb$, $\dots$; $\xi'_2=tsb$, $\xi'_4=ststb$, $\xi'_6=tststsb$, $\dots$; $\ell(\xi_{2i})=\ell(\xi'_{2i})=\ell(b)+2i$.
In this case, using a similar argument as in Case 4, we can get that $$
w=\underbrace{s\ltimes t\ltimes s\ltimes\dots}_{\text{$m/2$ factors}}\ltimes b=\underbrace{t\ltimes s\ltimes t\ltimes\dots}_{\text{$m/2$ factors}}\ltimes b .
$$

\begin{lem} \label{SGCM1} Let $W$ be the Weyl group of a Kac--Moody algebra $\mathfrak{g}$ corresponding to a symmetrizable generalized Cartan matrix. Let $\<-,-\>$ be the invariant bilinear form on $\mathfrak{h}^{\ast}$, where $\mathfrak{h}$ is the maximal toral subalgebra of $\mathfrak{g}$. Let $\alpha,\beta$ be two simple roots such that $\<\alpha,\beta\>\neq 0$. Suppose that $w\in W$, $w(\alpha)\in\{\pm\alpha\}$, $w(\beta)\in\{\pm\beta\}$. Then $w(\alpha)=\alpha$ if and only if $w(\beta)=\beta$; and $w(\alpha)=-\alpha$ if and only if $w(\beta)=-\beta$.
\end{lem}

\begin{proof} This follows easily from the equality $\<w(\alpha),w(\beta)\>=\<\alpha,\beta\>$.
\end{proof}

\bigskip

\section{Weyl groups of type $D_n$}

In this section we study the braid $I_\ast$-transformations between reduced $I_\ast$-expressions of involutions in the Weyl group $W(D_n)$ of type $D_n$. We shall
identify in Definition \ref{braid1} a finite set of basic braid $I_\ast$-transformations which span and preserve the sets of reduced $I_\ast$-expressions for
any involution in $W(D_n)$ for all $n$ simultaneously, and show in Theorem \ref{mainthm0} that any two reduced $I_\ast$-expressions for an involution in $W(D_n)$
 can be transformed into each other through a series of basic braid $I_\ast$-transformations.
\smallskip

Let $W(D_n)$ be the Weyl group of type $D_n$. It is generated by the simple reflections $\{s_u,s_1,\cdots,s_{n-1}\}$ which satisfy the following relations:            $$
\begin{aligned}
&s_u^2=1=s_i^2,\,\,\,\, for\,\,1\leq i\leq n-1,\\
&s_u s_2 s_u=s_2 s_u s_2,\\
&s_u s_1= s_1 s_u,\\
&s_i s_{i+1} s_i= s_{i+1} s_i s_{i+1},\,\,\,\, for\,\,1\leq i\leq n-2,\\
&s_u s_i= s_i s_u,\,\,\,\,for\,\,3\leq i\leq n-1,\\
&s_is_j=s_js_i,\,\,\,\,for\,\,1\leq i<j-1\leq n-2.
\end{aligned}
$$

Alternatively, $W(D_n)$ can be (cf. \cite{BB, HuJ2}) realized as the subgroup of the permutations on the set $
\{1,-1,2,-2,\cdots,n,-n\}$ such that: \begin{equation}\label{sign0}
\text{$\sigma(i)=j$ if and only if $\sigma(-i)=-j$ for any $i,j$, and $\#\{1\leq i\leq n|\sigma(i)<0\}$ is even.}
\end{equation}
In particular, under this identification, we have that $$
s_u=(1,-2)(-1,2)\,\,\,\text{and}\,\,\, s_i=(i,i+1)(-i,-i-1),\,\,\,for\,\,1\leq i<n .
$$
The subgroup generated by $s_1,s_2,\cdots,s_{n-1}$ (or $s_u,s_2,\cdots,s_{n-1}$) can be identified
with the symmetric group $\Sym_n$. Let $\tau$ be the automorphism of $W(D_n)$ which fixes each generator $s_i$ for
$2\leq i<n$ and exchanges the generators $s_1$ and $s_u$.

Let $\eps_1,\cdots,\eps_n$ be the standard basis of $\mathbb{R}^n$. We set $u:=\eps_1+\eps_2$ and $\alpha_i:=\eps_{i+1}-\eps_{i}$ for each $1\leq i<n$. For each $1\leq i\leq n$, we define $\eps_{-i}:=-\eps_i$. Then $W(D_n)$ acts on the set
$\{\eps_i|i=-n,\cdots,-2,-1,1,2,\cdots,n\}$ via $\sigma(\eps_i):=\eps_{\sigma(i)}$. Let $$
\Phi:=\{\pm\eps_i\pm\eps_j|1\leq i<j\leq n\}\,\,\text{and}\,\,E:=\text{$\mathbb{R}$-Span$\{v|v\in\Phi\}$}.
$$
Then $\Phi$ is the root system of type $D_n$ in $E$ with $W(D_n)$ being its Weyl group. We choose $\Delta:=\{u,\alpha_{i}|1\leq i<n\}$ to be the set of the simple
roots. Then $
\Phi^{+}=\{\eps_j\pm\eps_i|1\leq i<j\leq n\}$ is the set of positive roots. For any $0\neq \alpha\in E$, we write $\alpha>0$ if
$\alpha=\sum_{\beta\in\Delta}k_{\beta}\beta$ with $k_{\beta}\geq 0$ for each $\beta$.

For any $w\in W(D_n)$ and $\alpha\in\Delta$, it is well-known that \begin{equation}
\label{reflection0}
ws_{\alpha}w^{-1}=s_{w(\alpha)},
\end{equation}
where $s_{w(\alpha)}$ is the reflection with respect to hyperplane which is orthogonal to $w(\alpha)$.

\begin{lem} \label{length1} Let $w\in W(D_n)$ and $1\leq i<n$. Then

1) $ws_i<w$ if and only if $w(\eps_{i+1}-\eps_{i})<0$;

2) $ws_u<w$ if and only if $w(\eps_1+\eps_2)<0$.

\end{lem}

\begin{lem} \label{64a} Let $W=W(D_n)$ be the Weyl group of type $D_n$. Let $w\in I_\ast$ be an involution, and let $s=s_{\alpha}$ and $t=s_{\beta}$ for some
$\alpha\neq\beta$ in $\Delta$ with $$(\alpha,\beta)\in\{(\alpha_i,\alpha_{i+1}),(\alpha_{i+1},\alpha_{i}),(\alpha_2,u),(u,\alpha_{2})|1\leq i<n-1\}.
$$
Assume that $s,t$ are both descents of $w$ and let $b\in I_\ast$ be the unique minimal length representative of $W_K wW_K$ where $K:=\<s,t\>$. Assume further that
$b$ has no descents which commute with both $s$ and $t$. It then holds that $bs=sb$ and $bt=tb$ (Case 4 in the notation of Lemma \ref{7cases})
only if one of the following occurs: \begin{enumerate}
\item $b=1$;
\item $(\alpha,\beta)\in\{(\alpha_i,\alpha_{i+1}),(\alpha_{i+1},\alpha_i)\}$ and $b=s_{i+2}\ltimes s_{i+1}\ltimes s_{i}\ltimes d$, where $1\leq i<n-2$, $d\in I_\ast$ and $\rho(b)=\rho(d)+3$;
\item $(\alpha,\beta)\in\{(u,\alpha_2),(\alpha_2,u)\}$ and $b=s_{3}\ltimes s_2\ltimes s_{1}\ltimes s_u\ltimes d$, where $d\in I_\ast$ and $\rho(b)=\rho(d)+4$.
\end{enumerate}
\end{lem}

\begin{proof} By assumption, we have that $m_{st}=3$, $bsb=s$ and $btb=t$. It follows that $b(\alpha)=\pm\alpha, b(\beta)=\pm\beta$ (by (\ref{reflection0})).
By the expression of $w$ given in Case 4, both $(t, s, b)$ and $(s, t, b)$ are reduced $I_\ast$-sequences. Applying Corollary \ref{length0} and Corollary
\ref{deletion2}, we can deduce that $b(\alpha)>0<b(\beta)$. It follows that $b(\alpha)=\alpha, b(\beta)=\beta$. Without loss of generality, we can assume
that $(\alpha,\beta)\in\{(\alpha_i,\alpha_{i+1}),(u,\alpha_2)|1\leq i<n-1\}$.

Suppose that a) does not happen. Then there are only the following three possibilities:

\smallskip
{\it Case 1.} $(\alpha,\beta)=(\alpha_i,\alpha_{i+1})$ for some $1\leq i<n-1$. Then $b(i)=i, b(i+1)=i+1, b(i+2)=i+2$. By Lemma \ref{length1} and the assumption that a) does not happen we can deduce that \begin{equation}\label{b01}
b(1)<b(2)<\dots<b(i-1)\,\,\,\text{and}\,\,\,b(i+3)<b(i+4)<\dots<b(n)
\end{equation}
and $b\neq 1$.

Suppose that b) does not happen. We claim that $b(i+3)\geq i+3$. In fact, if $b(i+3)<i+3$ then we can deduce that
$b(i+3)< i$. In this case, $$
b^{-1}(\alpha_{i+2})=b(\alpha_{i+2})=b(\eps_{i+3})-b(\eps_{i+2})=b(\eps_{i+3})-\eps_{i+2}<0 ,
$$
and $b(\alpha_{i+2})\neq\pm\alpha_{i+2}$, which implies that $bs_{i+2}\neq s_{i+2}b$. By Corollary \ref{length0}, we get that $s_{i+2}\ltimes b=s_{i+2}bs_{i+2}$ and $\rho(b)=\rho(s_{i+2}\ltimes b)+1$.

Now, $$
(s_{i+2}bs_{i+2})^{-1}(\alpha_{i+1})=(s_{i+2}bs_{i+2})(\eps_{i+2}-\eps_{i+1})=s_{i+2}b(\eps_{i+3})-\eps_{i+1}<0 ,
$$
and $(s_{i+2}bs_{i+2})(\alpha_{i+1})\neq\pm\alpha_{i+1}$ implies that $(s_{i+2}bs_{i+2})s_{i+1}\neq s_{i+1}(s_{i+2}bs_{i+2})$. It follows from Corollary \ref{length0} that $$
s_{i+1}\ltimes (s_{i+2}\ltimes b)=s_{i+1}s_{i+2}bs_{i+2}s_{i+1}\,\,\,\text{and}\,\,\,\rho(b)=\rho(s_{i+1}\ltimes s_{i+2}\ltimes b)+2.
$$

Finally, $$
(s_{i+1}s_{i+2}bs_{i+2}s_{i+1})^{-1}(\alpha_{i})=(s_{i+1}s_{i+2}bs_{i+2}s_{i+1})(\eps_{i+1}-\eps_{i})=s_{i+1}s_{i+2}b(\eps_{i+3})-\eps_{i}<0 .
$$
It follows from Corollary \ref{length0} that $$
\rho(b)=\rho(s_{i}\ltimes s_{i+1}\ltimes s_{i+2}\ltimes b)+3.
$$
Set $d:=s_{i}\ltimes s_{i+1}\ltimes s_{i+2}\ltimes b$. Then $b=s_{i+2}\ltimes s_{i+1}\ltimes s_{i}\ltimes d$ and this is b) which contradicts our assumption.
This completes the proof of our claim.

According to (\ref{b01}) and the fact that $b(j)=j$ for any $i\leq j\leq i+2$, our above claim implies that $b(k)=k$ for any $i\leq k \leq n$.

If $ i=1$, then $ b=1$, a contradiction to our assumption that a) does not hold.

If $ i=2$, then by (\ref{sign0}) $ b(1)=1$ and hence $ b=1$, a contradiction to our assumption that a) does not hold.

It remains to consider the case when $i\geq 3$. In this case, since $s_u$ commutes with any $s_j$ when $j\geq 3$, our assumption implies that $s_ub>b$. Thus
we must have that $b(u)=b(\eps_1+\eps_2)>0$. Combining the claim with (\ref{sign0}) and (\ref{b01}), one can deduce that $b=1$, a contradiction to our assumption that a) does not hold. This completes the proof in Case 1.

\smallskip
{\it Case 2.} $(\alpha,\beta)=(u,\alpha_{2})$. Then $b(1)=1, b(2)=2, b(3)=3$. By Lemma \ref{length1} and the assumption that a) does not happen we can get that \begin{equation}\label{b02}
b(4)<b(5)<\dots<b(n) .
\end{equation}

By assumption, $b\neq 1$. It follows that $b(4)<4$ and hence by (\ref{sign0}) $b(4)\leq -4$ (because $b(j)=j$ for any $1\leq j\leq 3$). Furthermore, by (\ref{sign0}), we deduce that $b(5)<0$ and it follows that $b(5)\leq -4$ and hence $b(4)\leq -5$. In this case, $b^{-1}(\alpha_{3})=b(\alpha_{3})=b(\eps_4)-b(\eps_{3})=b(\eps_{4})-\eps_{3}<0$,
and $b(\alpha_{3})\neq\pm\alpha_{3}$, which implies that $bs_{3}\neq s_{3}b$. By Corollary \ref{length0}, we get that $s_{3}\ltimes b=s_{3}bs_{3}$ and $\rho(b)=\rho(s_{3}\ltimes b)+1$. In a similar way, we can get that
$$\begin{aligned}
& s_{2}\ltimes (s_{3}\ltimes b)=s_2s_{3}bs_{3}s_2\,\,\,\text{and}\,\,\, \rho(b)=\rho(s_2\ltimes s_{3}\ltimes b)+2,\\
& s_{1}\ltimes s_{2}\ltimes s_{3}\ltimes b=s_{1}s_2s_{3}bs_{3}s_2s_{1}\,\,\,\text{and}\,\,\,\rho(b)=\rho(s_{1}\ltimes s_2\ltimes s_{3}\ltimes b)+3.
\end{aligned}$$
Using the inequality $b(4)\leq -5$ we can also get that $(s_{1}\ltimes s_{2}\ltimes s_{3}\ltimes b)(\eps_1+\eps_2)<0$, and $$
s_u\ltimes s_{1}\ltimes s_{2}\ltimes s_{3}\ltimes b=s_us_{1}s_2s_{3}bs_{3}s_2s_{1}s_u\,\,\,\text{and}\,\,\, \rho(b)=\rho(s_u\ltimes s_{1}\ltimes s_2\ltimes s_{3}\ltimes b)+4.
$$
We set $d:=s_{u}\ltimes s_1\ltimes s_{2}\ltimes s_3\ltimes b$. Then $b=s_{3}\ltimes s_2\ltimes s_{1}\ltimes s_u\ltimes d$ and this is Case c). This completes the proof in Case 2 and hence finishes the proof of the lemma.
\end{proof}

\begin{rem} \label{rem11} Recall that $\tau$ is an automorphism of $W(D_n)$ which fixes each generator $s_i$ for $2\leq i<n$ and exchanges the generators
$s_1$ and $s_u$. In view of this automorphism $\tau$, the careful readers might ask why the case when $(\alpha,\beta)\in\{(\alpha_1,\alpha_{2}),(\alpha_{2},
\alpha_1)\}$ and
$b=s_{3}\ltimes s_{2}\ltimes s_{u}\ltimes s_1\ltimes d$ (where $d\in I_\ast$ and $\rho(b)=\rho(d)+4$) does not appear in Lemma \ref{64a}.
In fact, since $s_{3}\ltimes s_{2}\ltimes \underbrace{s_{u}\ltimes s_1}\ltimes d=s_{3}\ltimes s_{2}\ltimes \underbrace{s_{1}\ltimes s_u}\ltimes d$, this
``missing case" is actually included in Case b).
\end{rem}

\begin{lem} \label{67a} Let $W=W(D_n)$ be the Weyl group of type $D_n$. Let $w\in I_\ast$ be an involution, and let $s=s_{\alpha}$ and $t=s_{\beta}$ for some
$\alpha\neq\beta$ in $\Delta$ with $$(\alpha,\beta)\in\{(\alpha_i,\alpha_{i+1}),(\alpha_{i+1},\alpha_{i}),(\alpha_2,u),(u,\alpha_{2})|1\leq i<n-1\}.
$$
Assume that $s,t$ are both descents of $w$ and let $b\in I_\ast$ be the unique minimal length representative of $W_K wW_K$ where $K:=\<s,t\>$. Assume further that
$b$ has no descents which commute with both $s$ and $t$. It then holds that $bt=sb$ and $bs=tb$ (Case 6 in the notation of Lemma \ref{7cases}) only if one of the
following occurs: \begin{enumerate}
\item $b=1$;
\item $(\alpha,\beta)\in\{(\alpha_i,\alpha_{i+1}),(\alpha_{i+1},\alpha_i)\}$ and $b=s_{i-1}\ltimes s_{i}\ltimes s_{i+1}\ltimes d$, where $2\leq i<n-1$, $d\in I_\ast$ and $\rho(b)=\rho(d)+3$;
\item $(\alpha,\beta)\in\{(\alpha_2,\alpha_3),(\alpha_3,\alpha_2)\}$ and $b=s_{4}\ltimes s_3\ltimes s_{2}\ltimes d$, where $d\in I_\ast$ and $\rho(b)=\rho(d)+3$;
\item $(\alpha,\beta)\in\{(\alpha_2,\alpha_3),(\alpha_3,\alpha_2)\}$ and $b=s_{u}\ltimes s_2\ltimes s_{3}\ltimes d$, where $d\in I_\ast$ and $\rho(b)=\rho(d)+3$;
\item $(\alpha,\beta)\in\{(\alpha_1,\alpha_2),(\alpha_2,\alpha_1)\}$ and $b=s_{3}\ltimes s_2\ltimes s_{1}\ltimes s_u\ltimes d$, where $d\in I_\ast$ and $\rho(b)=\rho(d)+4$;
\item $(\alpha,\beta)\in\{(u,\alpha_2),(\alpha_2,u)\}$ and $b=s_{3}\ltimes s_2\ltimes s_{u}\ltimes s_1\ltimes d$, where $d\in I_\ast$ and $\rho(b)=\rho(d)+4$;
\item $(\alpha,\beta)\in\{(\alpha_2,\alpha_3),(\alpha_3,\alpha_2)\}$, $b=s_u\ltimes s_1\ltimes s_2\ltimes s_u\ltimes s_1\ltimes s_3$ and $\rho(b)=6$;
\item $(\alpha,\beta)\in\{(\alpha_1,\alpha_2),(\alpha_2,\alpha_1)\}$, $b=s_u\ltimes s_3\ltimes s_2\ltimes s_1\ltimes s_u\ltimes s_3$ and $\rho(b)=6$;
\item $(\alpha,\beta)\in\{(u,\alpha_2),(\alpha_2,u)\}$, $b=s_1\ltimes s_3\ltimes s_2\ltimes s_u\ltimes s_1\ltimes s_3$ and $\rho(b)=6$;
\item $(\alpha,\beta)\in\{(\alpha_2,\alpha_3),(\alpha_3,\alpha_2)\}$, $b=s_u\ltimes s_1\ltimes s_2\ltimes s_4\ltimes s_3\ltimes s_u\ltimes s_2\ltimes s_u\ltimes s_1\ltimes s_4 $ and $\rho(b)=10$.
\end{enumerate}
\end{lem}

\begin{proof} By assumption, we have that $m_{st}=3$, $bsb=t$ and $btb=s$. It follows that $b(\alpha)=\pm\beta, b(\beta)=\pm\alpha$ (by (\ref{reflection0})). Note that $b(\alpha)=\pm\beta$ if and only if  $b(\beta)=\pm\alpha$ because $b^2=1$. Without loss of generality, we can assume that $(\alpha,\beta)\in\{(\alpha_i,\alpha_{i+1}),(u,\alpha_2)|1\leq i<n-1\}$.

By the expression of $w$ given in Case 6 in the notation of Lemma \ref{7cases}, $s\ltimes t\ltimes b$ is a reduced $I_\ast$-sequence. Applying Corollary \ref{length0} and Corollary \ref{deletion2}, we can deduce that $b(\alpha)>0$. It follows that $b(\alpha)=\beta$.

Suppose that a) does not happen. There are only the following possibilities:

\smallskip
{\it Case 1.} $(\alpha,\beta)=(\alpha_i,\alpha_{i+1})$ for some $3\leq i<n-1$. Then $b(i)=-(i+2), b(i+2)=-i, b(i+1)=-(i+1)$. By Lemma \ref{length1} and our assumption, we can deduce that $s_jb>b, s_ub>b$ for any $1\leq j<i-1$ or $i+3\leq j<n$. Therefore, \begin{equation}\label{b1}
b(1)<b(2)<\dots<b(i-1),\,\,b(1)+b(2)>0\,\,\,\text{and}\,\,\,b(i+3)<b(i+4)<\dots<b(n) .
\end{equation}
Note that (\ref{b1}) implies that $b(i-1)>0$. Thus $$
b^{-1}(\alpha_{i-1})=b(\alpha_{i-1})=b(\eps_i)-b(\eps_{i-1})=-\eps_{i+2}-b(\eps_{i-1})<0 .
$$
Furthermore, $b(\alpha_{i-1})\neq\pm\alpha_{i-1}$ implies that $bs_{i-1}\neq s_{i-1}b$. By Corollary \ref{length0}, we get that $s_{i-1}\ltimes b=s_{i-1}bs_{i-1}$ and $\rho(b)=\rho(s_{i-1}\ltimes b)+1$.

Now, $$
(s_{i-1}bs_{i-1})^{-1}(\alpha_i)=(s_{i-1}bs_{i-1})(\eps_{i+1}-\eps_i)=-\eps_{i+1}-s_{i-1}b(\eps_{i-1})<0 ,
$$
and $(s_{i-1}bs_{i-1})(\alpha_{i})\neq\pm\alpha_{i}$ implies that $(s_{i-1}bs_{i-1})s_{i}\neq s_{i}(s_{i-1}bs_{i-1})$. It follows from Corollary \ref{length0} that $$
s_{i}\ltimes (s_{i-1}\ltimes b)=s_is_{i-1}bs_{i-1}s_i\,\,\,\text{and}\,\,\, \rho(b)=\rho(s_i\ltimes s_{i-1}\ltimes b)+2.
$$

Finally, $$
(s_is_{i-1}bs_{i-1}s_i)^{-1}(\alpha_{i+1})=(s_is_{i-1}bs_{i-1}s_i)(\eps_{i+2}-\eps_{i+1})=-\eps_{i-1}-s_is_{i-1}b(\eps_{i-1})<0 ,
$$
and $(s_is_{i-1}bs_{i-1}s_i)(\alpha_{i+1})\neq\pm\alpha_{i+1}$ implies that $$(s_is_{i-1}bs_{i-1}s_i)s_{i+1}\neq s_{i+1}(s_is_{i-1}bs_{i-1}s_i).$$ It follows from Corollary \ref{length0} that $$
s_{i+1}\ltimes (s_{i}\ltimes (s_{i-1}\ltimes b)=s_{i+1}s_is_{i-1}bs_{i-1}s_is_{i+1}\,\,\,\text{and}\,\,\, \rho(b)=\rho(s_{i+1}\ltimes s_i\ltimes s_{i-1}\ltimes b)+3.
$$
Set $d:=s_{i+1}\ltimes s_i\ltimes s_{i-1}\ltimes b$. Then $b=s_{i-1}\ltimes s_i\ltimes s_{i+1}\ltimes d$ and this is b) as required.

\smallskip
{\it Case 2.} $(\alpha,\beta)=(\alpha_2,\alpha_3)$. Then $b(2)=-4, b(4)=-2, b(3)=-3$. By Lemma \ref{length1} we can get that \begin{equation}\label{b2}
b(5)<b(6)<\dots<b(n) .
\end{equation}

Suppose that $b(5)<-1$. Then we must have that $b(5)\leq -5$ as $\{-2,-3,-4\}=\{b(3),b(4),b(2)\}$.

Suppose that $b(5)\leq -6$. In this case, $b^{-1}(\alpha_4)=b(\eps_5-\eps_4)=b(\eps_5)+\eps_2<0$ and $b(\alpha_4)\neq\pm\alpha_4$. Therefore $bs_4\neq s_4b$. By Corollary \ref{length0}, we get that $s_{4}\ltimes b=s_{4}bs_{4}$ and $\rho(b)=\rho(s_{4}\ltimes b)+1$. Next, $$(s_{4}bs_{4})^{-1}(\alpha_{3})=(s_{4}bs_{4})(\eps_4-\eps_{3})=s_4b(\eps_5)+\eps_3<0, $$
and $(s_{4}bs_{4})(\alpha_{3})\neq\pm\alpha_{3}$ (because $b(5)\leq -6$). This implies that $(s_{4}bs_{4})s_{3}\neq s_{3}(s_{4}bs_{4})$. It follows from Corollary \ref{length0} that $$
s_{3}\ltimes (s_{4}\ltimes b)=s_3s_{4}bs_{4}s_3\,\,\,\text{and}\,\,\, \rho(b)=\rho(s_3\ltimes s_{4}\ltimes b)+2.
$$
Finally, $$
(s_3s_{4}bs_{4}s_3)^{-1}(\alpha_{2})=(s_3s_{4}bs_{4}s_3)(\eps_{3}-\eps_{2})=s_3s_{4}b(\eps_{5})+\eps_{5}<0 ,
$$
and $(s_3s_{4}bs_{4}s_3)(\alpha_{2})\neq\pm\alpha_{2}$ implies that $$(s_3s_{4}bs_{4}s_3)s_{2}\neq s_{2}(s_3s_{4}bs_{4}s_3).$$ It follows from Corollary \ref{length0} that $$
s_{2}\ltimes (s_{3}\ltimes (s_{4}\ltimes b)=s_{2}s_3s_{4}bs_{4}s_3s_{2}\,\,\,\text{and}\,\,\, \rho(b)=\rho(s_{2}\ltimes s_3\ltimes s_{4}\ltimes b)+3.
$$
Set $d:=s_{4}\ltimes s_3\ltimes s_{2}\ltimes b$. Then $b=s_{4}\ltimes s_3\ltimes s_{2}\ltimes d$ and this is c) as required.

Suppose that $b(5)=-5$, then by (\ref{b2}) and the fact that $b^2=1$ we can deduce that $b(1)\in\{\pm 1,\pm 6\}$. Using a similar argument as in the last paragraph, we can prove that if $b(1)=6$ then $b=s_{1}\ltimes s_2\ltimes s_{3}\ltimes d$ with $\rho(b)=\rho(d)+3$; while if $b(1)=-6$ then $b=s_{u}\ltimes s_2\ltimes s_{3}\ltimes d$ with $\rho(b)=\rho(d)+3$. These are b) and d) respectively.

If $b(1)=-1$, then by (\ref{b2}) we can further deduce that $b(i)=i$ for any $i\geq 6$. However, this is impossible by (\ref{sign0}).

If $b(1)=1$, then by (\ref{b2}) again we can deduce that $b(i)=i$ for any $i\geq 6$. In this subcase, note that $s_us_1$ maps $1,2$ to $-1,-2$ respectively
and fixes any $j\in\{3,4,5\}$; while for each $2\leq k\leq 4$, $s_k\cdots s_2s_us_1s_2\cdots s_k$ maps $1, k+1$ to $-1,-(k+1)$ respectively and fixes any $j\in\{1,2,\cdots,5\}\setminus\{1,k+1\}$. We can deduce that $$\begin{aligned}
b&=s_2s_3s_2(s_3s_2s_us_1s_2s_3)(s_4s_3s_2s_us_1s_2s_3s_4)(s_2s_us_1s_2)(s_us_1)\\
&=s_us_1s_2s_3s_4s_3s_2s_us_1s_2s_3s_4s_us_1s_2s_us_1\\
&=s_u\ltimes s_1\ltimes s_2\ltimes s_4\ltimes s_3\ltimes s_u\ltimes s_2\ltimes s_u\ltimes s_1\ltimes s_4 ,
\end{aligned}
$$
and $\rho(b)=10$, where the last equality follows from a brute-force calculation. As a result, we get j) as required.

Suppose that $b(5)=-1$. Then $b(u)=b(\eps_1+\eps_2)=-\eps_5-\eps_4<0$ and $b(u)\neq\pm u$. Therefore $bs_u\neq s_ub$. By Corollary \ref{length0}, we get that $s_{u}\ltimes b=s_{u}bs_{u}$ and $\rho(b)=\rho(s_{u}\ltimes b)+1$. Next, $$(s_{u}bs_{u})^{-1}(\alpha_{2})=(s_{u}bs_{u})(\eps_3-\eps_{2})=-\eps_3-\eps_5<0, $$
and $(s_{u}bs_{u})(\alpha_{2})\neq\pm\alpha_{2}$. This implies that $(s_{u}bs_{u})s_{2}\neq s_{2}(s_{u}bs_{u})$. It follows from Corollary \ref{length0} that $$
s_{2}\ltimes (s_{u}\ltimes b)=s_2s_{u}bs_{u}s_2\,\,\,\text{and}\,\,\,\rho(b)=\rho(s_2\ltimes s_{u}\ltimes b)+2.
$$
Now we have that $(s_2s_{u}bs_{u}s_2)(\alpha_3)=(s_2s_{u}bs_{u}s_2)(\eps_4-\eps_3)=\eps_1-\eps_5<0$, and $(s_2s_{u}bs_{u}s_2)(\alpha_{3})\neq\pm\alpha_{3}$ implies that $$(s_2s_{u}bs_{u}s_2)s_{3}\neq s_{3}(s_2s_{u}bs_{u}s_2).$$ It follows from Corollary \ref{length0} that $$
s_{3}\ltimes (s_{2}\ltimes (s_{u}\ltimes b)=s_{3}s_2s_{u}bs_{u}s_2s_{3}\,\,\,\text{and}\,\,\, \rho(b)=\rho(s_{3}\ltimes s_2\ltimes s_{u}\ltimes b)+3.
$$
Set $d:=s_{3}\ltimes s_2\ltimes s_{u}\ltimes b$. Then $b=s_{u}\ltimes s_2\ltimes s_{3}\ltimes d$ and this is d) as required.

Therefore, it remains to consider the case when $b(5)>0$. It follows from (\ref{b2}) and the fact that $b^2=1$ that $b(1)=-1$ and $b(i)=i$ for any $5\leq i\leq n$. We can deduce that $$
b=(s_us_1)(s_2s_us_1s_2)(s_3s_2s_us_1s_2s_3)s_2s_3s_2=s_u\ltimes s_1\ltimes s_2\ltimes s_u\ltimes s_1\ltimes s_3,
$$
and $\rho(b)=6$, which is g) as required.

\smallskip
{\it Case 3.} $(\alpha,\beta)=(\alpha_1,\alpha_2)$. Then $b(1)=-3, b(3)=-1, b(2)=-2$. By Lemma \ref{length1} again, we can get that \begin{equation}\label{b3}
b(4)<b(5)<\dots<b(n) .
\end{equation}

Applying (\ref{b3}) and (\ref{sign0}), we see that $b(4)<0$ in this case. Then we must have that $b(4)\leq -4$ as $\{-1,-2,-3\}=\{b(1),b(2),b(3)\}$. Assume that $b(4)\leq -5$. Then $b(\alpha_3)=b(\eps_4-\eps_3)=b(\eps_4)+\eps_1<0$ and $b(\alpha_3)\neq\pm\alpha_3$. Therefore $bs_3\neq s_3b$. By Corollary \ref{length0}, we get that $s_{3}\ltimes b=s_{3}bs_{3}$ and $\rho(b)=\rho(s_{3}\ltimes b)+1$. Next, $$(s_{3}bs_{3})^{-1}(\alpha_{2})=(s_{3}bs_{3})(\eps_3-\eps_{2})=s_3b(\eps_4)+\eps_2<0, $$
and $(s_{3}bs_{3})(\alpha_{2})\neq\pm\alpha_{2}$ (because $b(4)\leq -5$). This implies that $(s_{3}bs_{3})s_{2}\neq s_{2}(s_{3}bs_{3})$. It follows from Corollary \ref{length0} that $$
s_{2}\ltimes (s_{3}\ltimes b)=s_2s_{3}bs_{3}s_2\,\,\,\text{and}\,\,\, \rho(b)=\rho(s_2\ltimes s_{3}\ltimes b)+2.
$$
By a similar argument, we can get that $$\begin{aligned}
&s_{1}\ltimes s_{2}\ltimes s_{3}\ltimes b=s_{1}s_2s_{3}bs_{3}s_2s_{1}\,\,\,\text{and}\,\,\,\rho(b)=\rho(s_{1}\ltimes s_2\ltimes s_{3}\ltimes b)+3,\\
&s_{u}\ltimes s_1\ltimes s_{2}\ltimes s_{3}\ltimes b=s_us_{1}s_2s_{3}bs_{3}s_2s_{1}s_u\,\,\,\text{and}\,\,\, \rho(b)=\rho(s_u\ltimes s_{1}\ltimes s_2\ltimes s_{3}\ltimes b)+4 .
\end{aligned}
$$
Set $d:=s_u\ltimes s_{1}\ltimes s_2\ltimes s_{3}\ltimes b$. Then $b=s_{3}\ltimes s_2\ltimes s_{1}\ltimes s_u\ltimes d$ and this is e) as required.

It remains to consider the case when $b(4)=-4$. In this case, by (\ref{b3}) again, we can deduce that $b(j)=j$ for any $j\geq 5$. It follows that $$
b=(s_us_1)(s_2s_us_1s_2)(s_3s_2s_us_1s_2s_3)s_1s_2s_1=s_u\ltimes s_3\ltimes s_2\ltimes s_1\ltimes s_u\ltimes s_3,
$$
and $\rho(b)=6$, which is h) as required.


\smallskip
{\it Case 4.} $(\alpha,\beta)=(u,\alpha_2)$. Then $b(1)=3, b(3)=1, b(2)=-2$. By Lemma \ref{length1} again, we can get that \begin{equation}\label{b4}
b(1)+b(2)>0\,\,\,\text{and}\,\,\,b(4)<b(5)<\dots<b(n) .
\end{equation}

Applying (\ref{b4}) it is easy to see that $b(4)<0$ in this case. Then we must have that $b(4)\leq -4$ as $\{-1,-2,-3\}=\{b(-1),b(2),b(-3)\}$. Assume that $b(4)\leq -5$. Then $b(\alpha_3)=b(\eps_4-\eps_3)=b(\eps_4)-\eps_1<0$ and $b(\alpha_3)\neq\pm\alpha_3$. Therefore $bs_3\neq s_3b$. By Corollary \ref{length0}, we get that $s_{3}\ltimes b=s_{3}bs_{3}$ and $\rho(b)=\rho(s_{3}\ltimes b)+1$. Next, $$
(s_{3}bs_{3})^{-1}(\alpha_{2})=(s_{3}bs_{3})(\eps_3-\eps_{2})=s_3b(\eps_4)+\eps_2<0, $$
and $(s_{3}bs_{3})(\alpha_{2})\neq\pm\alpha_{2}$ (because $b(4)\leq -5$). This implies that $(s_{3}bs_{3})s_{2}\neq s_{2}(s_{3}bs_{3})$. It follows from Corollary \ref{length0} that $$
s_{2}\ltimes (s_{3}\ltimes b)=s_2s_{3}bs_{3}s_2\,\,\,\text{and}\,\,\, \rho(b)=\rho(s_2\ltimes s_{3}\ltimes b)+2.
$$
By a similar argument, we can get that $$\begin{aligned}
&s_{u}\ltimes s_{2}\ltimes s_{3}\ltimes b=s_{u}s_2s_{3}bs_{3}s_2s_{u}\,\,\,\text{and}\,\,\, \rho(b)=\rho(s_{u}\ltimes s_2\ltimes s_{3}\ltimes b)+3,\\
&s_1\ltimes s_{u}\ltimes s_{2}\ltimes s_{3}\ltimes b=s_1s_{u}s_2s_{3}bs_{3}s_2s_{u}s_1\,\,\,\text{and}\,\,\,\rho(b)=\rho(s_1\ltimes s_{u}\ltimes s_2\ltimes s_{3}\ltimes b)+4.
\end{aligned}
$$
Set $d:=s_1\ltimes s_{u}\ltimes s_2\ltimes s_{3}\ltimes b$. Then $b=s_{3}\ltimes s_2\ltimes s_{u}\ltimes s_1\ltimes d$ and this is f) as required..

It remains to consider the case when $b(4)=-4$. In this case, by (\ref{b4}) again, we can deduce that $b(j)=j$ for any $j\geq 5$. It follows that $$
b=(s_us_1)(s_3s_2s_us_1s_2s_3)s_1s_2s_1=s_1\ltimes s_3\ltimes s_2\ltimes s_u\ltimes s_1\ltimes s_3,
$$
and $\rho(b)=6$, which is i) as required.
This completes the proof of the lemma.
\end{proof}

\begin{rem} Note that the possible values of $b$ in Lemma \ref{67a} are preserved by the automorphism $\tau$ of $W(D_n)$.
For example, for j) in Lemma \ref{67a}, we actually have that $$
s_u\ltimes s_1\ltimes s_2\ltimes s_4\ltimes s_3\ltimes s_u\ltimes s_2\ltimes s_u\ltimes s_1\ltimes s_4=s_1\ltimes s_u\ltimes s_2\ltimes s_4\ltimes s_3\ltimes s_1\ltimes s_2\ltimes s_1\ltimes s_u\ltimes s_4 .
$$
To see this, it suffices to show that $s_u\ltimes s_2\ltimes s_u\ltimes s_1\ltimes s_4=s_1\ltimes s_2\ltimes s_1\ltimes s_u\ltimes s_4$. In fact, we have that $$\begin{aligned}
&\quad\,s_u\ltimes s_2\ltimes s_u\ltimes s_1\ltimes s_4=\underbrace{s_2\ltimes s_u\ltimes s_2}\ltimes s_1\ltimes s_4=s_4\ltimes s_2\ltimes s_u\ltimes s_2\ltimes s_1\\
&=s_4\ltimes s_2\ltimes s_u\ltimes \underbrace{s_1\ltimes s_2}=s_4\ltimes s_2\ltimes \underbrace{s_1\ltimes s_u}\ltimes s_2=s_4\ltimes s_2\ltimes s_1\ltimes \underbrace{s_2\ltimes s_u}\\
&=s_4\ltimes \underbrace{s_1\ltimes s_2\ltimes s_1}\ltimes s_u=s_1\ltimes s_2\ltimes s_1\ltimes s_u\ltimes s_4 .\end{aligned}
$$
\end{rem}

\begin{rem} \label{rem0} We consider reduce $I_\ast$-expressions for involutions in the Weyl group of type $D_n$. In this case, in addition to the basic braid
$I_\ast$-transformations given by Proposition \ref{braid}, one clearly has to add the following natural ``right end transformations": $$
\begin{aligned}
& s_i\ltimes s_{i+1}\longleftrightarrow s_{i+1}\ltimes s_i,\,\,\,s_2\ltimes s_{u}\longleftrightarrow s_{u}\ltimes s_2,\\
& s_j\ltimes s_{u}\longleftrightarrow s_{u}\ltimes s_j,\,\,\,s_k\ltimes s_{l}\longleftrightarrow s_{l}\ltimes s_k,
\end{aligned}
$$
where $1\leq i<n-1$, $1\leq j,k,l<n$, $j\neq 2$, $|k-l|>1$. Given the result \cite[Definition 2.12, Theorem 3.1]{HZH1} for the type $A$ case, it is tempting to
speculate that for involutions in $W(D_n)$ these are all the basic braid $I_\ast$-transformation that we need. However, it turns out that this is {\it NOT} the
case. In fact, one has to add one extra transformation in the case of type $D_n$ (see the last transformation in Definition \ref{braid1}), which is a new
phenomenon for type $D_n$.
\end{rem}

\begin{dfn}\label{braid1} By a basic braid $I_\ast$-transformation, we mean one of the following transformations and their inverses: $$
\begin{aligned}
1)\,\, &(\cdots,s_j,s_{j+1},s_j,\cdots)\longmapsto (\cdots,s_{j+1},s_j,s_{j+1},\cdots) ,\\
2)\,\, &(\cdots,s_u,s_{2},s_u,\cdots)\longmapsto (\cdots,s_2,s_u,s_2,\cdots) ,\\
3)\,\, & (\cdots,s_b,s_{c},\cdots)\longmapsto (\cdots,s_{i_a},s_{c},s_b,\cdots) ,\\
4)\,\, & (\cdots,s_d,s_u,\cdots)\longmapsto (\cdots,s_u,s_d,\cdots) ,\\
5)\,\, & (\cdots,s_k,s_{k+1})\longmapsto (\cdots,s_{k+1},s_k) ,\\
6)\,\, & (\cdots,s_2,s_u)\longmapsto (\cdots,s_u,s_2) ,\\
7)\,\, & (\cdots,s_2,s_3,s_u,s_1,s_2,s_u,s_1,s_3)\longmapsto (\cdots,s_3,s_2,s_u,s_1,s_2,s_u,s_1,s_3) ,
\end{aligned}
$$
where all the sequences appearing above are reduced sequences, and the entries marked by corresponding ``$\cdots$" must match, and in the first two
transformations (i.e., 1) and 2)) we further require that the right end part entries marked by ``$\cdots$" must be non-empty. We define a braid
$I_\ast$-transformation to be the composition of a series of basic braid $I_\ast$-transformations.
\end{dfn}

Let $w\in I_\ast$ and $s_{i_1},\cdots,s_{i_k}\in S$. By definition, it is clear that $(s_{i_1},\cdots,s_{i_k},w)$ is a reduced sequence if and only if $(s_{i_1},\cdots,s_{i_k},s_{j_1},\cdots,s_{j_t})$ is a reduced sequence for some (and any) reduced $I_\ast$-expression $(s_{j_1},\cdots,s_{j_t})$ of $w$.

\begin{dfn} \label{braid3} Let $(s_{i_1},\cdots,s_{i_k},w), (s_{j_1},\cdots,s_{j_l},u)$ be two reduced $I_\ast$-sequences, where $w,u\in I_{\ast}$. We shall write $(s_{i_1},\cdots,s_{i_k},w)\longleftrightarrow (s_{j_1},\cdots,s_{j_l},u)$ whenever there exists a series of braid $I_\ast$-transformations which transform $$(s_{i_1},\cdots,s_{i_k},s_{l_1},\cdots,s_{l_b})$$ into $(s_{j_1},\cdots,s_{j_l},s_{p_1},\cdots,s_{p_c})$, where $(s_{l_1},\cdots,s_{l_b})$ and $(s_{p_1},\cdots,s_{p_c})$ are some reduced $I_\ast$-expressions of $w$ and $u$ respectively. Moreover, we shall also write $$
(i_1,\cdots,i_k)\longleftrightarrow (j_1,\cdots,j_k)
$$
whenever $(s_{i_1},\cdots,s_{i_k})\longleftrightarrow (s_{j_1},\cdots,s_{j_k})$.
\end{dfn}

\begin{thm} \label{braid0} Let $(s_{i_1},\cdots,s_{i_k}), (s_{j_1},\cdots,s_{j_k})$ be two reduced $I_\ast$-sequences which can be transformed into each other through a series of basic braid $I_{\ast}$-transformations. Then $$
s_{i_1}\ltimes s_{i_2}\ltimes\dots\ltimes s_{i_k}=s_{j_1}\ltimes s_{j_2}\ltimes\dots\ltimes s_{j_k} .
$$
\end{thm}

\begin{proof} In fact, this follows easily from Proposition \ref{braid} and Remark \ref{rem0} except for the last transformation (i.e., 7)) in Definition \ref{braid1}. For that one, one can use a brutal-force calculation to check that \begin{equation}\label{dd}
\text{$\begin{matrix}s_2\ltimes s_3\ltimes s_u\ltimes s_1\ltimes s_2\ltimes s_u\ltimes s_1\ltimes s_3=s_2s_3s_us_1s_2s_us_1s_3s_2s_1s_us_3\\
=s_3\ltimes s_2\ltimes s_u\ltimes s_1\ltimes s_2\ltimes s_u\ltimes s_1\ltimes s_3.\end{matrix}$}
\end{equation}
This completes the proof of the theorem.
\end{proof}

A well-known classical fact of Matsumoto (\cite{Mat}) says that any two reduced expressions for an element in any Weyl group can be transformed into each other through a series of basic braid transformations. In Theorem \ref{braid0} we have shown that any basic braid $I_\ast$-transformations on reduced $I_\ast$-expression for a given $w\in I_\ast$ do not change the involution $w$ itself. The following theorem says something more than this.

\begin{thm} \label{mainthm0} Let $w\in I_\ast$. Then any two reduced $I_\ast$-expressions for $w$ can be transformed into each other through a series of basic braid $I_\ast$-transformations.
\end{thm}

\begin{proof} We prove the theorem by induction on $\rho(w)$. Suppose that the theorem holds for any $w\in I_\ast$  with $\rho(w)\leq k$. Let $w\in I_\ast$ with $\rho(w)=k+1$. Let $(s_{i_0},s_{i_1},s_{i_2},\cdots,s_{i_k})$ and $(s_{j_0},s_{j_1},s_{j_2},\cdots,s_{j_k})$ be two reduced $I_\ast$-expressions for $w\in I_\ast$. We need to prove that \begin{equation}\label{goal}
(i_0,i_1,\cdots,i_k)\longleftrightarrow (j_0,j_1,\cdots,j_k).
\end{equation}
For simplicity, we set $s=s_{i_0}, t=s_{j_0}$. Let $m$ be the order of $st$.

If $m=3$, then we are in the situations of Cases 1,2,3,4,6 of Lemma \ref{7cases}. Suppose that we are in Cases 1,2,3 of Lemma \ref{7cases}. Then we get that $$
s_{i_0}\ltimes s_{i_1}\ltimes\cdots\ltimes s_{i_k}\longleftrightarrow s\ltimes t\ltimes s\ltimes b\longleftrightarrow t\ltimes s\ltimes t\ltimes b\longleftrightarrow s_{j_0}\ltimes s_{j_1}\ltimes\cdots\ltimes s_{j_k},
$$
where the first and the third ``$\longleftrightarrow$" follows from induction hypothesis, and the second ``$\longleftrightarrow$" follows from the expression of $w$ given in Cases 1,2,3 of Lemma \ref{7cases}. It remains to consider Cases 4,6 of Lemma \ref{7cases}. To this end, we shall apply Lemmas \ref{64a} and \ref{67a}. In these cases, if $b=1$, then $k=1$ and $s_{i_0}\ltimes s_{i_1}=s\ltimes t\longleftrightarrow t \ltimes s=s_{j_0}\ltimes s_{j_1}$.
Henceforth, we assume that $b\neq 1$.

Our strategy is as follows: in order to prove (\ref{goal}), it suffices to show that \begin{equation}\label{strategy1}
\text{$(i_0,i_1,\cdots,i_k)\longleftrightarrow (s_\alpha,\dots)$ and
$(j_0,j_1,\cdots,j_k)\longleftrightarrow (s_\alpha,\dots)$ for some $s_\alpha\in S$.}
\end{equation}
Once this is proved, then  (\ref{goal}) follows from induction
hypothesis.

With the this in mind, our task is reduced to the verification of (\ref{strategy1}). In fact, (\ref{strategy1}) is easy to verify except for
Cases g), h), i), j) in Lemma \ref{67a}. Suppose that we are in Case b) of Lemma \ref{64a}. Without
loss of generality, we assume that $i_0=i$ and $j_0=i+1$. Then $$\begin{aligned}
&(s_{j_0},s_{j_1},\cdots,s_{j_k})\longleftrightarrow (s_{i+1},s_i,b)\longleftrightarrow(s_{i+1},s_i,\underbrace{s_{i+2},s_{i+1},s_i,d})\longleftrightarrow\\
&\qquad (s_{i+1},\underbrace{s_{i+2},s_i},s_{i+1},s_i,d)\longleftrightarrow(s_{i+1},s_{i+2},\underbrace{s_{i+1},s_{i},s_{i+1}},d)\\
&\qquad \longleftrightarrow(\underbrace{s_{i+2},s_{i+1},s_{i+2}},s_{i},s_{i+1},d),
\end{aligned}
$$
and $$\begin{aligned}
&(s_{i_0},s_{i_1},\cdots,s_{i_k})\longleftrightarrow (s_i,s_{i+1},b)\longleftrightarrow(s_i,s_{i+1},\underbrace{s_{i+2},s_{i+1},s_i,d})\longleftrightarrow\\
&\qquad (s_{i},\underbrace{s_{i+2},s_{i+1},s_{i+2}},s_i,d)\longleftrightarrow(\underbrace{s_{i+2},s_{i}},s_{i+1},s_{i+2},s_{i},d),
\end{aligned}
$$
as required, where in both equalities, the first ``$\longleftrightarrow$" follows from induction hypothesis.

Suppose that we are in Case c) of Lemma \ref{64a}. Without loss of generality, we assume that $i_0=u$ and $j_0=2$. Then $$\begin{aligned}
&(s_{i_0},s_{i_1},\cdots,s_{i_k})\longleftrightarrow (s_u,s_2,b)\longleftrightarrow(s_u,s_2,\underbrace{s_3,s_2,s_1,s_u,d})
\longleftrightarrow\\
&\qquad (s_u,\underbrace{s_3,s_2,s_3},s_1,s_u,d)\longleftrightarrow (\underbrace{s_3,s_u},s_2,s_3,s_1,s_u,d),\\
\end{aligned}
$$
and $$\begin{aligned}
&(s_{j_0},s_{j_1},\cdots,s_{j_k})\longleftrightarrow (s_2,s_u,b)\longleftrightarrow(s_2,s_u,\underbrace{s_3,s_2,s_1,s_u,d})
\longleftrightarrow\\
&\qquad (s_2,\underbrace{s_3,s_u},s_2,\underbrace{s_u,s_1},d)\longleftrightarrow (s_2,s_3,\underbrace{s_2,s_u,s_2},s_1,d),\\
&\qquad \longleftrightarrow (\underbrace{s_3,s_2,s_3},s_u,s_2,s_1,d),
\end{aligned}
$$
as required.
By a similar argument, we can prove (\ref{goal}) if we are in Cases a),b),c),d),e),f) of Lemma \ref{67a}.

Suppose that we are in Case g) of Lemma \ref{67a}. Without loss of generality, we assume that $i_0=2$ and $j_0=3$. Then we have that $$\begin{aligned}
&(s_{i_0},s_{i_1},\cdots,s_{i_k})\longleftrightarrow (s_2, s_3, s_u, s_1, s_2, s_u, s_1, s_3)\longleftrightarrow (s_3, s_2, s_u, s_1, s_2, s_u, s_1, s_3)\\
&\qquad \longleftrightarrow (s_{j_0},s_{j_1},\cdots,s_{j_k}),
\end{aligned}$$
where the first and the third ``$\longleftrightarrow$" follows from induction hypothesis and the second ``$\longleftrightarrow$" follows from the extra
transformation 7) in Definition \ref{braid1}.

Suppose that we are in Case h) of Lemma \ref{67a}. Without loss of generality, we assume that $i_0=2$ and $j_0=1$. Then we have that $$\begin{aligned}
&\quad\,(s_{i_0},s_{i_1},\cdots,s_{i_k})\longleftrightarrow (s_2, s_1,b)\longleftrightarrow (s_2, s_1, \underbrace{s_u, s_3, s_2, s_1, s_u, s_3})\\
&\qquad\longleftrightarrow (s_2, \underbrace{s_3, s_u, s_1}, s_2, \underbrace{s_u, s_1}, s_3)\longleftrightarrow (\underbrace{s_3, s_2, s_u, s_1, s_2, s_u, s_1, s_3})\\
&\qquad\longleftrightarrow (s_3, s_2, \underbrace{s_1, s_u}, s_2, s_u, s_1, s_3)\longleftrightarrow (s_3, s_2, s_1, \underbrace{s_2, s_u, s_2}, s_1, s_3)\\
&\qquad\longleftrightarrow (s_3, \underbrace{s_1, s_2, s_1}, s_u, s_2, s_1, s_3)\longleftrightarrow (\underbrace{s_1, s_3}, s_2, s_1, s_u, s_2, s_1, s_3)\\
&\qquad\longleftrightarrow (s_1, s_3, s_2, \underbrace{s_u, s_1}, s_2, s_1, s_3)\longleftrightarrow (s_1, s_3, s_2, s_u, \underbrace{s_2, s_1, s_2}, s_3)\\
&\qquad\longleftrightarrow (s_1, s_3, s_2, s_u, s_2, s_1, \underbrace{s_3, s_2})\longleftrightarrow (s_1, s_3, s_2, s_u, s_2, \underbrace{s_3, s_1}, s_2)\\
&\qquad\longleftrightarrow (s_1, s_3, s_2, s_u, s_2, s_3, \underbrace{s_2, s_1})\longleftrightarrow (s_1, s_3, s_2, s_u, \underbrace{s_3, s_2, s_3}, s_1)\\
&\qquad\longleftrightarrow (s_1, s_3, s_2, \underbrace{s_3, s_u}, s_2, s_3, s_1)\longleftrightarrow (s_1, \underbrace{s_2, s_3, s_2}, s_u, s_2, s_3, s_1)\\
&\qquad\longleftrightarrow (s_1, s_2, s_3, \underbrace{s_u, s_2, s_u}, s_3, s_1)\longleftrightarrow (s_1, s_2, \underbrace{s_u, s_3}, s_2, s_u,
\underbrace{s_1, s_3})\\
&\qquad\longleftrightarrow (s_1, s_2, s_u, s_3, s_2, \underbrace{s_1, s_u},
s_3)\longleftrightarrow (s_{j_0},s_{j_1},\cdots,s_{j_k}),
\end{aligned}
$$
as required, where the fourth  ``$\longleftrightarrow$" follows from the extra transformation 7) in Definition \ref{braid1}.

Suppose that we are in Case i) of Lemma \ref{67a}. Without loss of generality, we assume that $i_0=2$ and $j_0=u$. Then we have that $$\begin{aligned}
&\quad\,(s_{i_0},s_{i_1},\cdots,s_{i_k})\longleftrightarrow (s_2,s_u,b)\longleftrightarrow (s_2, s_u, \underbrace{s_1, s_3, s_2, s_u, s_1, s_3})\\
&\qquad\longleftrightarrow (s_2, \underbrace{s_3, s_u, s_1}, s_2, s_u, s_1, s_3)\longleftrightarrow (\underbrace{s_3, s_2, s_u, s_1, s_2, s_u, s_1, s_3})\\
&\qquad\longleftrightarrow (s_3, s_2, s_u, s_1, s_2, \underbrace{s_1, s_u}, s_3)\longleftrightarrow (s_3, s_2, s_u, \underbrace{s_2, s_1, s_2}, s_u, s_3)\\
&\qquad\longleftrightarrow  (s_3, \underbrace{s_u, s_2, s_u}, s_1, s_2, s_u, s_3)\longleftrightarrow(\underbrace{s_u, s_3}, s_2, s_u, s_1, s_2, s_u, s_3)\\
&\qquad\longleftrightarrow (s_u, s_3, s_2, \underbrace{s_1,s_u}, s_2, s_u, s_3)\longleftrightarrow (s_u, s_3, s_2, s_1,\underbrace{s_2, s_u, s_2}, s_3)\\
&\qquad\longleftrightarrow (s_u, s_3, s_2, s_1, s_2, s_u, \underbrace{s_3, s_2})\longleftrightarrow (s_u, s_3, s_2, s_1, s_2, \underbrace{s_3, s_u}, s_2)\\
&\qquad\longleftrightarrow (s_u, s_3, s_2, s_1, s_2, s_3, \underbrace{s_2, s_u})\longleftrightarrow (s_u, s_3, s_2, s_1, \underbrace{s_3, s_2, s_3}, s_u)\\
&\qquad\longleftrightarrow (s_u, s_3, s_2, \underbrace{s_3, s_1}, s_2, s_u, s_3)\longleftrightarrow (s_u, \underbrace{s_2, s_3, s_2}, s_1, s_2, s_u, s_3)\\
&\qquad\longleftrightarrow (s_u, s_2, s_3, \underbrace{s_1, s_2, s_1}, s_u, s_3)\longleftrightarrow (s_u, s_2, \underbrace{s_1, s_3}, s_2, \underbrace{s_u, s_1}, s_3)\\
&\qquad\longleftrightarrow (s_{j_0},s_{j_1},\cdots,s_{j_k}), \end{aligned}
$$
as required, where the fourth  ``$\longleftrightarrow$" follows from the extra transformation 7) in Definition \ref{braid1}..

Suppose that we are in Case j) of Lemma \ref{67a}. Without loss of generality, we assume that $i_0=2$ and $j_0=3$. Then we have that
$k=11$, and $$\begin{aligned}
&\quad\,\,(i_0,i_1,\dots,i_{11})\longleftrightarrow (s_2, s_3, b)\longleftrightarrow (s_2, s_3, \underbrace{s_u, s_1, s_2, s_4, s_3, s_u, s_2, s_u, s_1, s_4})\\
&\longleftrightarrow
(s_2, s_3, \underbrace{s_4, s_u, s_1, s_2}, s_3, \underbrace{s_2, s_u, s_2}, s_1, s_4)\longleftrightarrow
(s_2, s_3, s_4, s_u, s_1, \underbrace{s_3, s_2, s_3}, s_u, s_2, s_1, s_4)\\
&\longleftrightarrow
(s_2, s_3, s_4, \underbrace{s_3, s_u, s_1}, s_2, s_3, s_u, s_2, s_1, s_4)\longleftrightarrow
(s_2, \underbrace{s_4, s_3, s_4}, s_u, s_1, s_2, s_3, s_u, s_2, s_1, s_4)\\
&\longleftrightarrow
(\underbrace{s_4, s_2}, s_3, s_4, s_u, s_1, s_2, s_3, s_u, s_2, s_1, s_4) .
\end{aligned}
$$
On the other hand, using the definition of the braid $I_\ast$-transformations, we can get that $$\begin{aligned}
&\quad\,(j_0,j_1,\dots,j_{11})\longleftrightarrow (s_3, s_2, b)\longleftrightarrow (s_3, s_2, \underbrace{s_u, s_1, s_2, s_4, s_3, s_u, s_2, s_u, s_1, s_4})\\
&\longleftrightarrow (s_3, s_2, \underbrace{s_1, s_u,} s_2, s_4, s_3, \underbrace{s_4, s_u, s_2, s_u, s_1})\longleftrightarrow
(s_3, s_2, s_1, s_u, s_2, \underbrace{s_3, s_4, s_3}, s_u, s_2, s_u, s_1)\\
&\longleftrightarrow
(s_3, s_2, s_1, s_u, s_2, \underbrace{s_u, s_3, s_4, s_3}, s_2, s_u, s_1)\longleftrightarrow
(s_3, s_2, s_1, \underbrace{s_2, s_u, s_2}, s_3, s_4, s_3, s_2, s_u, s_1)\\
&\longleftrightarrow
(s_3, \underbrace{s_1, s_2, s_1}, s_u, s_2, s_3, s_4, s_3, s_2, s_u, s_1)\longleftrightarrow
(\underbrace{s_1, s_3}, s_2, s_1, s_u, s_2, s_3, s_4, s_3, s_2, s_u, s_1)\\
&\longleftrightarrow
(s_1, s_3, s_2, s_1, s_u, s_2, \underbrace{s_4, s_3, s_4}, s_2, s_u, s_1)\longleftrightarrow
(s_1, s_3, \underbrace{s_4, s_2, s_1, s_u, s_2}, s_3, s_4, s_2, s_u, s_1)\\
&\longleftrightarrow
(s_1, s_3, s_4, s_2, s_1, s_u, s_2, s_3, \underbrace{s_2, s_4}, s_u, s_1)\longleftrightarrow
(s_1, s_3, \underbrace{s_2, s_4}, s_1, s_u, \underbrace{s_3, s_2, s_3}, \underbrace{s_u, s_1, s_4})\\
&\longleftrightarrow
(s_1, s_3, s_2, s_4, \underbrace{s_3, s_1, s_u}, s_2, \underbrace{s_u, s_1, s_3}, s_4)\longleftrightarrow
(s_1, s_3, s_2, s_4, s_3, s_1, s_u, s_2, s_u, s_1, \underbrace{s_4, s_3})\\
&\longleftrightarrow
(s_1, s_3, s_2, s_4, s_3, \underbrace{s_4, s_1, s_u, s_2, s_u, s_1}, s_3)\longleftrightarrow
(s_1, s_3, s_2, \underbrace{s_3, s_4, s_3}, s_1, s_u, s_2, s_u, s_1, s_3)\\
&\longleftrightarrow
(s_1, \underbrace{s_2, s_3, s_2}, s_4, s_3, s_1, s_u, s_2, s_u, s_1, s_3)\longleftrightarrow
(s_1, s_2, s_3, \underbrace{s_4, s_2}, s_3, \underbrace{s_u, s_1}, s_2, s_u, s_1, s_3)\\
&\longleftrightarrow
(s_1, s_2, s_3, s_4, \underbrace{s_3, s_2, s_u, s_1, s_2, s_u, s_1, s_3})\longleftrightarrow
(s_1, s_2, \underbrace{s_4, s_3, s_4}, s_2, s_u, s_1, s_2, s_u, s_1, s_3)\\
&\longleftrightarrow
(\underbrace{s_4, s_1, s_2}, s_3, s_4, s_2, s_u, s_1, s_2, s_u, s_1, s_3) .
\end{aligned}
$$
So again we are in a position to apply the induction hypothesis. This completes the proof of (\ref{goal}) when $m=3$. As a result, we can make the following useful
observation:

\smallskip
{\it Observation 1.} If there exists some $s_\alpha, s_\beta\in S$ such that $s_\alpha s_\beta$ has order $1$ or $3$ and $(i_0,i_1,\dots,i_k)
\longleftrightarrow
(s_\alpha,\dots)$, $(j_0,j_1,\dots,j_k)\longleftrightarrow (s_\beta,\dots)$, then by the result we have obtained,we can deduce that $(i_0,i_1,\dots,i_k)\longleftrightarrow
(j_0,j_1,\dots,j_k)$.
\smallskip

Henceforth we assume that $m=2$. That is, $st=ts$. By Lemma \ref{square}, $\rho(s_{j_0}\ltimes w)=\rho(s_{j_1}\ltimes s_{j_2}\ltimes\cdots\ltimes s_{j_k})=k<k+1$.
It follows from Lemma \ref{rankfunc} that $\ell(s_{j_0}w)=\ell(w)-1$.
Equivalently, $$
\ell(s_{j_0}(s_{i_0}\ltimes s_{i_1}\ltimes s_{i_2}\ltimes\cdots\ltimes s_{i_k}))=\ell(s_{i_0}\ltimes s_{i_1}\ltimes s_{i_2}\ltimes\cdots\ltimes s_{i_k})-1 .
$$
Applying  Lemma \ref{rankfunc} again, we can deduce that $$
\rho(s_{j_0}\ltimes(s_{i_0}\ltimes s_{i_1}\ltimes s_{i_2}\ltimes\cdots\ltimes s_{i_k}))=k .
$$
Applying Proposition \ref{exchange}, we get that $$
s_{j_0}\ltimes(s_{i_0}\ltimes s_{i_1}\ltimes s_{i_2}\ltimes\cdots\ltimes s_{i_k})=s_{i_0}\ltimes s_{i_1}\ltimes s_{i_2}\ltimes\cdots
\ltimes s_{i_{a-1}}\ltimes s_{i_{a+1}}\ltimes\cdots\ltimes s_{i_k}
$$
for some $0\leq a\leq k$. In particular, $s_{i_0}\ltimes s_{i_1}\ltimes s_{i_2}\ltimes\cdots
\ltimes s_{i_{a-1}}\ltimes s_{i_{a+1}}\ltimes\cdots\ltimes s_{i_k}=s_{j_1}\ltimes\cdots\ltimes s_{j_k}$.

Since $$
s_{j_0}\ltimes s_{i_0}\ltimes s_{i_1}\ltimes s_{i_2}\ltimes\cdots
\ltimes s_{i_{a-1}}\ltimes s_{i_{a+1}}\ltimes\cdots\ltimes s_{i_k}=s_{i_0}\ltimes s_{i_1}\ltimes s_{i_2}\ltimes\cdots\ltimes s_{i_k},
$$
it is clear that $(j_0,i_0,i_1,i_2,\cdots,i_{a-1},i_{a+1},\cdots,i_k)$ is a reduced $I_\ast$-expression for $w$.

It remains to show that \begin{equation}\label{goal2} (j_0,i_0,i_1,i_2,\cdots,i_{a-1},i_{a+1},\cdots,i_k)\longleftrightarrow
(i_0,i_1,i_2,\cdots,i_k).\end{equation}
In fact, by induction hypothesis, $$(i_0,i_1,i_2,\cdots,i_{a-1},i_{a+1},\cdots,i_k)\longleftrightarrow (j_1,j_2,\cdots,j_k)$$ because $s_{i_0}\ltimes
s_{i_1}\ltimes s_{i_2}\ltimes\cdots \ltimes s_{i_{a-1}}\ltimes s_{i_{a+1}}\ltimes\cdots\ltimes s_{i_k}=s_{j_1}\ltimes\cdots\ltimes s_{j_k}$.
Once (\ref{goal2}) is proved, we can deduce that $(j_0,i_0,i_1,i_2,\cdots,i_{a-1},i_{a+1},\cdots,i_k)\longleftrightarrow (j_0,j_1,j_2,\cdots,j_k)$. Composing these
transformations, we prove (\ref{goal}).

If $a>0$, then as $st=ts$, $$
(j_0,i_0,i_1,i_2,\cdots,i_{a-1},i_{a+1},\cdots,i_k)\longleftrightarrow(i_0,j_0,i_1,i_2,\cdots,i_{a-1},i_{a+1},\cdots,i_k),
$$
and hence (\ref{goal2}) follows from Observation 1.

It remains to consider the case when $a=0$. In this case, (\ref{goal2}) becomes \begin{equation}\label{goal3} (j_0,i_1,i_2,\cdots,i_k)\longleftrightarrow
(i_0,i_1,i_2,\cdots,i_k),\end{equation}
We set $w_1:=s_{i_1}\ltimes s_{i_2}\ltimes\dots\ltimes s_{i_k}$. Then $w_1\neq 1$. There are two possibilities:

\smallskip
{\it Case 1.} $(s,t)\in\{(s_{i_0},s_{j_0}),(s_{j_0},s_{i_0})|1\leq i_0<j_0-1<n-1\}$. Without loss of generality, we assume that
$(s,t)=(s_{i_0},s_{j_0})$, where $i_0<j_0-1<n-1$.

Note that $i_1\not\in\{i_0,j_0\}$ because both $(i_0,i_1,\cdots)$ and $(j_0,i_1,\cdots)$ are reduced $I_\ast$-sequences. We can assume that either $|i_1-i_0|=1$ or $|i_1-j_0|=1$ because otherwise $$
(j_0,i_1,\cdots)\longleftrightarrow(i_1,j_0,\cdots) \longleftrightarrow(i_1,i_0,\cdots)\longleftrightarrow(i_0,i_1,\cdots).
$$
Without loss of generality we assume that $|i_1-i_0|=1$. Suppose that $j_0-i_0>2$. Then we must have that $|i_1-j_0|>1$. It follows that $$
(j_0,i_1,\cdots)\longleftrightarrow(i_1,j_0,\cdots) \longleftrightarrow(i_0,i_1,\cdots),
$$
where the second ``$\longleftrightarrow$" follows from Observation 1. Therefore, it suffices to consider the case when $j_0=i_0+2$. Furthermore, by a similar
argument,
we can consider only the subcase when $i_1=i_0+1$ and hence $w_1=s_{i_0+1}\ltimes w_2$ with $\rho(w_1)=\rho(w_2)+1$. Since $s_{i_0}\ltimes s_{i_0+1}\neq s_{i_0+2}\ltimes s_{i_0+1}$, it follows that $w_2\neq 1$.

Suppose that $i_0\geq 2$. If $i_2=i_0-1$ or $i_2=u$ and $i_0=2$, then $$\begin{aligned}
(j_0,i_1,i_2,\cdots)&=(i_0+2,i_0+1,i_2,\cdots)\longleftrightarrow (i_2,i_0+2,i_0+1,\cdots)\\
&\qquad  \longleftrightarrow (i_0,i_1,i_2,\cdots),
\end{aligned}$$
where the last ``$\longleftrightarrow$" follows from Observations 1. If $i_2<i_0-1$ or $i_2>i_0+3$, or $i_2=u$ and $i_0>2$, then $$\begin{aligned}
(j_0,i_1,i_2,\cdots)&=(i_0+2,i_0+1,i_2,\cdots)\longleftrightarrow (i_2,i_0+2,i_0+1,\cdots) \\
&\qquad\longleftrightarrow (i_2,i_0,i_1,\cdots)\longleftrightarrow
(i_0,i_1,i_2,\cdots),
\end{aligned}$$

Similarly, if $i_2=i_0+3$  then $$\begin{aligned}
(i_0,i_1,i_2,\cdots)&=(i_0,i_0+1,i_2,\cdots)\longleftrightarrow(i_2,i_0,i_0+1,\cdots)=(i_0+3,i_0,i_0+1,\cdots)\\
&\qquad \longleftrightarrow (i_0+2,i_1,i_2,\cdots)=(j_0,i_1,i_2,\cdots),
\end{aligned}$$
where the last ``$\longleftrightarrow$" follows from Observations 1. Since $(i_0,i_0+1,i_2,\cdots)$ is a reduced $I_\ast$-sequence, it is clear that $i_2\neq i_0+1$. It
remains to consider the case when $i_2\in\{i_0,i_0+2\}$. If $i_2=i_0$, then $$\begin{aligned}
(i_0,i_1,i_2,\cdots)&=(i_0,i_0+1,i_0,\cdots)\longleftrightarrow(i_0+1,i_0,i_0+1,\cdots)\\
&\qquad\longleftrightarrow (i_0+2,i_1,i_2,\cdots)=(j_0,i_1,i_2,\cdots).
\end{aligned}
$$
Similarly, if $i_2=i_0+2$, then $$\begin{aligned}
(j_0,i_1,i_2,\cdots)&=(i_0+2,i_0+1,i_0+2,\cdots)\longleftrightarrow (i_0+1,i_0+2,i_0+1,\cdots)\\
&\qquad  \longleftrightarrow (i_0,i_1,i_2,\cdots)=(i_0,i_1,i_2,\cdots).
\end{aligned}$$

Therefore, we can consider only the situation when $i_0=1$. By similar reasoning as before, we can assume that $i_2=u$, hence $w_2=s_{u}\ltimes w_3$ with
$\rho(w_2)=\rho(w_3)+1$. In particular, $w_3(u)>0$. Since $s_1\ltimes s_{2}\ltimes s_u\neq s_3\ltimes s_2\ltimes s_u$, it follows that $w_3\neq 1$.

Now $w_3^{-1}(\alpha)=w_3(\alpha)<0$, for $\alpha\in\Delta$, only if $\alpha\in\{\alpha_i|i\geq 1\}$. By similar reasoning as before, we can consider only the case
when $s_3$ is a descent of $w_3$. Hence $w_3=s_{3}\ltimes w_4$ with $\rho(w_3)=\rho(w_4)+1$. However, in this case, $$
(j_0, i_1, i_2, \cdots)\longleftrightarrow(3,2,u,3,w_4)\longleftrightarrow (3,2,\underbrace{3,u},w_4)\longleftrightarrow (\underbrace{2,3,2},u,w_4) .
$$
Applying Observation 1 again, we see that $$
(2,3,2,u,w_4)\longleftrightarrow (1,2,u,3,w_4)\longleftrightarrow(i_0,i_1,i_2,\cdots). $$ As a result, (\ref{goal2}) follows at once.

\smallskip
{\it Case 2.} $(s,t)\in\{(s_i,s_u),(s_u,s_{i})|1\leq i<n, i\neq 2\}$.

If $i\geq 3$, then (\ref{goal2}) can be proved by applying the automorphism $\tau$ or using in the same argument as in Case 1. It remains to consider the case when
$i=1$. Without loss of generality, we assume that $(s,t)=(s_{i_0},s_{j_0})=(s_1,s_u)$.

Note that $s_t$ is not a descent of $w_1$ for $s_t\in\{s_1,s_u\}$ because both $(s_{i_0},s_{i_1},s_{i_2},\cdots)$ and $(s_{j_0},s_{i_1},s_{i_2},\cdots)$ are reduced.
If $s_t$ is a descent of $w_1$ for some $t>2$. That is, $w_1(\alpha_t)<0$ and hence $w_1=s_{t}\ltimes w_2$ with $\rho(w_1)=\rho(w_2)+1$. Then $$\begin{aligned}
&\quad\,(s_{i_0},s_{i_1},\cdots)\longleftrightarrow (s_1,w_1)\longleftrightarrow (s_1,s_t,w_2)\longleftrightarrow (s_t,s_1,w_2)\longleftrightarrow
(s_t,s_u,w_2)\\
&\longleftrightarrow
(s_u,s_t,w_2)\longleftrightarrow (s_u,w_1)\longleftrightarrow (s_{j_0},s_{i_1},s_{i_2},\cdots),
\end{aligned}$$
and we are done. Therefore, it suffices to consider the case when $s_2$ is a descent of $w_1$. That is, $w_1(\alpha_2)<0$ and hence
$w_1=s_{2}\ltimes w_2$ with $\rho(w_1)=\rho(w_2)+1$. Note that $s_{1}\ltimes s_2\neq s_u\ltimes s_2$. It follows that $w_2\neq 1$.

If $s_t$ is a descent of $w_2$ for some $t>3$. That is, $w_2(\alpha_t)<0$ and hence $w_2=s_{t}\ltimes w_3$ with $\rho(w_2)=\rho(w_3)+1$. Then $$\begin{aligned}
&\quad\,(s_{i_0},s_{i_1},s_{i_2},\cdots)\longleftrightarrow (s_1,w_1)\longleftrightarrow (s_1,s_2,w_2)\longleftrightarrow (s_1,s_2,s_t,w_3)\longleftrightarrow
(s_t,s_1,s_2,w_3)\\
&\longleftrightarrow
(s_t,s_u,s_2,w_3)\longleftrightarrow (s_u,s_2,s_t,w_3)\longleftrightarrow (s_u,s_2,w_2)\longleftrightarrow (s_u,w_1)\longleftrightarrow (s_{j_0},s_{i_1},s_{i_2},\cdots),
\end{aligned}$$
and we are done.

If $s_u$ is a descent of $w_2$. That is, $w_2(u)<0$ and hence $w_2=s_{u}\ltimes w_3$ with $\rho(w_2)=\rho(w_3)+1$. Then applying Observation 1 (in the fourth
``$\longleftrightarrow$"), $$\begin{aligned}
&\quad\,(s_{j_0},s_{i_1},s_{i_2},\cdots)\longleftrightarrow (s_u,s_2,w_2)\longleftrightarrow (s_u,s_2,s_u,w_3)\longleftrightarrow(s_2,s_u,s_2,w_3)\\
&\longleftrightarrow
(s_1,s_2,s_u,w_3)\longleftrightarrow (s_1,s_2,w_2)\longleftrightarrow (s_1,w_1)\longleftrightarrow (s_{i_0},s_{i_1},s_{i_2},\cdots),
\end{aligned}$$
and we are done. The same argument applies to the case when $s_1$ is a descent of $w_3$. Note that $s_2$ is not a descent of $w_2$ because
$(s_{i_0},s_{i_1},s_{i_2},\cdots)$ is reduced. Therefore, it suffices to consider the case when $s_3$ is a descent of $w_2$. That is, $w_2(\alpha_3)<0$ and hence
$w_2=s_{3}\ltimes w_3$ with $\rho(w_2)=\rho(w_3)+1$.  Note that $s_{1}\ltimes s_2\ltimes s_3\neq s_u\ltimes s_2\ltimes s_3$. It follows that $w_3\neq 1$.

Repeating this argument, we shall finally get that $$
w_1=s_2\ltimes s_3\ltimes s_4\ltimes\dots\ltimes s_{n-1}\ltimes w_{n-1},
$$
such that $\rho(w_1)=\rho(w_{n-1})+n-2$. By direct calculation, we can check that $$
s_1\ltimes s_2\ltimes s_3\ltimes s_4\ltimes\dots\ltimes s_{n-1}\neq s_u\ltimes s_2\ltimes s_3\ltimes s_4\ltimes\dots\ltimes s_{n-1}.
$$
It follows that $w_{n-1}\neq 1$. Therefore, $\{\alpha\in\Delta|w_{n-1}(\alpha)<0\}\neq\emptyset$. If $w_{n-1}(\alpha_j)<0$ for some $2\leq j<n-1$, then
$w_{n-1}=s_j\ltimes w_n$ with $\rho(w_{n-1})=\rho(w_n)+1$. Hence $$\begin{aligned}
&(s_{i_0},s_{i_1},s_{i_2},\cdots)\longleftrightarrow (s_1, s_2,\cdots, s_{n-1}, s_j, w_n)\\
&\quad \longleftrightarrow (s_1, s_2,\dots, s_{j}, s_{j+1}, \underbrace{s_j, s_{j+2},\dots, s_{n-1}}, w_n)\\
&\quad \longleftrightarrow (s_1, s_2,\dots, s_{j-1},\underbrace{s_{j+1}, s_j, s_{j+1}}, s_{j+2},\dots, s_{n-1}, w_n)\\
&\quad \longleftrightarrow (\underbrace{s_{j+1}, s_1, s_2,\dots, s_{j-1}}, s_j, s_{j+1}, s_{j+2},\dots, s_{n-1}, w_n)\\
&\quad \longleftrightarrow (s_{j+1}, \underbrace{s_u, s_2,\dots, s_{j-1}, s_j, s_{j+1}, s_{j+2},\dots, s_{n-1}, w_n})\\
&\quad \longleftrightarrow (\underbrace{s_u, s_2,\dots, s_{j-1}, s_{j+1}}, s_j, s_{j+1}, s_{j+2},\dots, s_{n-1}, w_n)\\
&\quad \longleftrightarrow (s_u, s_2,\dots, s_{j-1}, \underbrace{s_j, s_{j+1}, s_{j}}, s_{j+2},\dots, s_{n-1}, w_n)\\
&\quad \longleftrightarrow (s_u, s_2,\dots, s_{j-1}, s_j, \underbrace{s_{j+1}, s_{j+2},\dots, s_{n-1}, s_{j}}, w_n)\\
&\quad \longleftrightarrow (s_{j_0},s_{i_1},s_{i_2},\cdots),
\end{aligned}$$
as required. Using a similar argument together with Observation 1 one can prove (\ref{goal2}) when $w_{n-1}(\alpha_1)<0$ or $w_{n-1}(u)<0$. This completes the proof of the theorem.
\end{proof}

\bigskip
\section{Weyl groups of type $B_n$}

In this section we study the braid $I_\ast$-transformations between reduced $I_\ast$-expressions of involutions in the Weyl group $W(B_n)$ of type $B_n$.
We shall identify in Definition \ref{braid1B} a finite set of basic braid $I_\ast$-transformations which span and preserve the sets of reduced $I_\ast$-expressions
for any involution in $W(B_n)$ for all $n$ simultaneously, and show in Theorem \ref{mainthm0B} that any two reduced $I_\ast$-expressions for an involution in
$W(B_n)$ can be transformed into each other through a series of basic braid $I_\ast$-transformations.

Let $W(B_n)$ be the Weyl group of type $B_n$. It is generated by the simple reflections $\{s_0,s_1,\cdots,s_{n-1}\}$ which satisfy the following relations:            $$
\begin{aligned}
&s_i^2=1,\,\,\,\, for\,\,0\leq i\leq n-1,\\
&s_0 s_1 s_0 s_1=s_1 s_0 s_1 s_0,\\
&s_i s_{i+1} s_i= s_{i+1} s_i s_{i+1},\,\,\,\, for\,\,1\leq i\leq n-2,\\
&s_is_j=s_js_i,\,\,\,\,for\,\,0\leq i<j-1\leq n-2.
\end{aligned}
$$

Alternatively, $W(B_n)$ can be realized as the subgroup of the permutations on the set $\{1,-1,2,-2,\cdots,n,-n\}$ (cf. \cite{BB}) such that: \begin{equation}
\label{sign0B}
\begin{matrix} \text{$\sigma(i)=j$ if and only if $\sigma(-i)=-j$ for any $i,j$.}\end{matrix}
\end{equation}
In particular, under this identification, we have that $$
s_0=(1,-1),\quad\,\, s_i=(i,i+1)(-i,-i-1),\,\,\,for\,\,1\leq i<n .
$$

Let $\eps_1,\cdots,\eps_n$ be the standard basis of $\mathbb{R}^n$. We set $\alpha_0:=\eps_1$, $\alpha_i:=\eps_{i+1}-\eps_{i}$ for each $1\leq i<n$. For each $1\leq i\leq n$, we define $\eps_{-i}:=-\eps_i$. Then $W$ acts on the set
$\{\eps_i|i=-n,\cdots,-2,-1,1,2,\cdots,n\}$ via $\sigma(\eps_i):=\eps_{\sigma(i)}$. Let $$
\Phi:=\{\pm\eps_i\pm\eps_j|1\leq i<j\leq n\}\cup\{\pm\eps_i|1\leq i\leq n\},\,\,\,E:=\text{$\mathbb{R}$-Span$\{v|v\in\Phi\}$}.
$$
Then $\Phi$ is the root system of type $B_n$ in $E$ with $W(B_n)$ being its Weyl group. We choose $\Delta:=\{\alpha_{i}|0\leq i<n\}$ to be the set of the simple
roots. Then $
\Phi^{+}=\{\eps_j\pm\eps_i,|1\leq i<j\leq n\}\cup \{\eps_i|1\leq i\leq n\}$ is the set of positive roots. For any $0\neq \alpha\in E$, we write $\alpha>0$ if
$\alpha=\sum_{\beta\in\Delta}k_{\beta}\beta$ with $k_{\beta}\geq 0$ for each $\beta$.

For any $w\in W(B_n)$ and $\alpha\in\Delta$, it is well-known that \begin{equation}
\label{reflection0B}
ws_{\alpha}w^{-1}=s_{w(\alpha)},
\end{equation}
where $s_{w(\alpha)}$ is the reflection with respect to hyperplane which is orthogonal to $w(\alpha)$.

\begin{lem} \label{length1B} Let $w\in W(B_n)$ and $1\leq i<n$. Then

1) $ws_i<w$ if and only if $w(\eps_{i+1}-\eps_{i})<0$;

2) $ws_0<w$ if and only if $w(\eps_1)<0$.

\end{lem}

\begin{lem} \label{64bB} Let $W=W(B_n)$ be the Weyl group of type $B_n$. Let $w\in I_\ast$ be an involution, and let $s=s_{\alpha}$ and $t=s_{\beta}$ for some
$\alpha\neq\beta$ in $\Delta$ with $$(\alpha,\beta)\in\{(\alpha_0,\alpha_{1}),(\alpha_{1},\alpha_0)\}.
$$
Assume that $s,t$ are both descents of $w$ and let $b\in I_\ast$ be the unique minimal length representative of $W_K wW_K$ where $K:=\<s,t\>$. Assume further that
$b$ has no descents which commute with both $s$ and $t$. Then $m_{st}=4$ and $s_0b\neq bs_1$ (i.e., excluding Cases 4,6,7 in the notation of Lemma \ref{7cases}).
Moreover, it holds that $bs=sb$ and $bt=tb$ (Case 5 in the notation of Lemma \ref{7cases}) only if one of the following occurs: \begin{enumerate}
\item $b=1$;
\item $b=s_2\ltimes s_1\ltimes s_0$;
\item $b=s_2\ltimes s_1\ltimes s_0\ltimes s_1 \ltimes s_2 \ltimes d$, where $d\in I_\ast$ and $\rho(b)=\rho(d)+5$.
\end{enumerate}
\end{lem}

\begin{proof} Since $m=4$ is even, we see that Case 4 and Case 6 do not happen. The Case 7 happens if and only if $s_0 b=b s_1$, $s_1 b=b s_0$. Since
$\eps_2-\eps_1$ and $\eps_1$ are roots of different lengths, so $b(\eps_2-\eps_1)\neq \pm\eps_1$. Thus this case would not happen either.

It is clear that Case 5 happens only if $b(\alpha)=\pm\alpha, b(\beta)=\pm\beta$ (by \ref{reflection0B}). By the expression of $w$ given in Case 5 of Lemma
\ref{7cases}, both $(s_0, s_1, s_0, b)$ and $(s_1, s_0, s_1, b)$ are reduced $I_\ast$-sequences. Applying Corollary \ref{length0} and Corollary \ref{deletion2},
we can deduce that $b(\alpha)>0<b(\beta)$. It follows that $b(\alpha_0)=\alpha_0, b(\alpha_1)=\alpha_1$. So we have that $b(1)=1$, $b(2)=2$.

Suppose that a) does not happen, i.e, $b\neq 1$. By assumption, any $s_t$ with $t\geq 3$ is not a descent of $b$. It follows from Lemma \ref{length1B} that
$b(3)<b(4)<\dots<b(n)$. If $b(3)>0$ then $b(3)\geq 3$ (because $b(1)=1$ and $b(2)=2$) and it follows that $b(k)=k$ for any $k$, a contradiction. Therefore, we can assume that $b(3)<0$ and hence
$b(3)\leq -3$ by (\ref{sign0B}).

Suppose that $b(3)=-3$. Then we must have that $b(k)=k$ for any $k\geq 4$. We can deduce that  $$b=s_2s_1s_0s_1s_2= s_2\ltimes s_1\ltimes s_0,$$
which is b) as required.

It suffices to consider the case when $b(3)<-3$. In this case, $$
b^{-1}(\alpha_{2})=b(\alpha_{2})=b(\eps_3-\eps_{2})=b(\eps_{3})-\eps_{2}<0,
$$
and $b(\alpha_{2})\neq\pm\alpha_{2}$, which implies that $bs_{2}\neq s_{2}b$. By Corollary \ref{length0}, we get that $s_{2}\ltimes b=s_{2}bs_{2}$ and
$\rho(b)=\rho(s_{2}\ltimes b)+1$.

In a similar way, we have that
$$
\pm\alpha_{1}\neq(s_{2}bs_{2})(\alpha_1)=(s_{2}bs_{2})(\eps_{2}-\eps_1)
=s_{2}b(\eps_{3})-\eps_{1}<0,
$$
$$
\pm\alpha_{0}\neq(s_1s_{2}bs_{2}s_1)(\alpha_0)=(s_1s_{2}bs_{2}s_1)
(\eps_1)=s_1s_{2}b(\eps_{3})<0,
$$
$$
\pm\alpha_{1}\neq(s_0s_1s_{2}bs_{2}s_1s_0)(\alpha_1)=\eps_{2}+s_0s_1s_{2}b(\eps_{3})<0,
$$
$$
\pm\alpha_{2}\neq(s_1s_0s_1s_{2}bs_{2}s_1s_0s_1)(\alpha_2)
=\eps_{3}+s_1s_0s_1s_2b(\eps_{3})<0.
$$
It follows from Corollary \ref{length0} that$$s_{2}\ltimes s_{1}\ltimes s_{0}\ltimes s_1\ltimes s_2\ltimes b=s_2s_1s_0s_1s_2bs_2s_1s_0s_1s_2\,\,\text{and}$$
$$
\rho(b)=\rho(s_{2}\ltimes s_{1}\ltimes s_{0}\ltimes s_1\ltimes s_2\ltimes b)+5.
$$
Set $d:=s_{2}\ltimes s_{1}\ltimes s_{0}\ltimes s_1\ltimes s_2\ltimes b$. Then $b=s_{2}\ltimes s_{1}\ltimes s_{0}\ltimes s_1\ltimes s_2\ltimes d$ and this is c)
as required. This completes the proof of the lemma.
\end{proof}

\begin{lem} \label{64aB} Let $W=W(B_n)$ be the Weyl group of type $B_n$. Let $w\in I_\ast$ be an involution, and let $s=s_{\alpha}$ and $t=s_{\beta}$ for some
$\alpha\neq\beta$ in $\Delta$ with $$(\alpha,\beta)\in\{(\alpha_i,\alpha_{i+1}),(\alpha_{i+1},\alpha_{i})|1\leq i<n-1\}.
$$
Assume that $s,t$ are both descents of $w$ and let $b\in I_\ast$ be the unique minimal length representative of $W_K wW_K$ where $K:=\<s,t\>$. Assume further that
$b$ has no descents which commute with both $s$ and $t$. Then $m_{st}=3$. Moreover, it holds that $bs=sb$ and $bt=tb$ (Case 4 in the notation of Lemma
\ref{7cases}) only if one of the following occurs: \begin{enumerate}
\item $b=1$;
\item $b=s_{i-1}\ltimes s_{i}\ltimes s_{i+1}\ltimes d$, where $2\leq i<n-1$, $d\in I_\ast$ and $\rho(b)=\rho(d)+3$;
\item $b=s_{i+2}\ltimes s_{i+1}\ltimes s_{i}\ltimes d$, where $1\leq i<n-2$, $d\in I_\ast$ and $\rho(b)=\rho(d)+3$.
\end{enumerate}
\end{lem}
\begin{proof} By assumption, we must have that $n\geq 4$. Suppose that a) does not happen, i.e., $b\neq 1$. By assumption, we have that $bsb=s$ and $btb=t$. It follows that $b(\alpha)=\pm\alpha, b(\beta)=\pm\beta$ (by (\ref{reflection0B})).
By the expression of $w$ given in Case 4, both $(t, s, b)$ and $(s, t, b)$ are reduced $I_\ast$-sequences. Applying Corollary \ref{length0} and Corollary
\ref{deletion2}, we can deduce that $b(\alpha)>0<b(\beta)$. It follows that $b(\alpha)=\alpha, b(\beta)=\beta$.

Without loss of generality, we can assume that $(\alpha,\beta)=(\alpha_i,\alpha_{i+1})$ for some $1\leq i<n-1$. Then $b(i)=i, b(i+1)=i+1, b(i+2)=i+2$.
By Lemma \ref{length1} and the assumption that $s_t$ is not a descent of $w$ for any $t<i-1$ or $t\geq i+3$, we can deduce that \begin{equation}\label{b01B}
b(1)<b(2)<\dots<b(i-1)\,\,\,\text{and}\,\,\,b(i+3)<b(i+4)<\dots<b(n)
\end{equation}
and $b\neq 1$.

Suppose that $i\geq 2$. If $b(i-1)>i-1$ then we can get that $b(i-1)\geq i+3$. In this case, $$
b^{-1}(\alpha_{i-1})=b(\alpha_{i-1})=b(\eps_i)-b(\eps_{i-1})=\eps_{i}-b(\eps_{i-1})<0 ,
$$
and $b(\alpha_{i-1})\neq\pm\alpha_{i-1}$, which implies that $bs_{i-1}\neq s_{i-1}b$. By Corollary \ref{length0}, we get that $s_{i-1}\ltimes b=s_{i-1}bs_{i-1}$
and $\rho(b)=\rho(s_{i-1}\ltimes b)+1$.

Now, $$
(s_{i-1}bs_{i-1})^{-1}(\alpha_i)=(s_{i-1}bs_{i-1})(\eps_{i+1}-\eps_i)=\eps_{i+1}-s_{i-1}b(\eps_{i-1})<0 ,
$$
and $(s_{i-1}bs_{i-1})(\alpha_{i})\neq\pm\alpha_{i}$ implies that $(s_{i-1}bs_{i-1})s_{i}\neq s_{i}(s_{i-1}bs_{i-1})$. It follows from Corollary \ref{length0}
that $$
s_{i}\ltimes (s_{i-1}\ltimes b)=s_is_{i-1}bs_{i-1}s_i,\quad \rho(b)=\rho(s_i\ltimes s_{i-1}\ltimes b)+2.
$$

Finally, $$
(s_is_{i-1}bs_{i-1}s_i)^{-1}(\alpha_{i+1})=(s_is_{i-1}bs_{i-1}s_i)(\eps_{i+2}-\eps_{i+1})=\eps_{i+2}-s_is_{i-1}b(\eps_{i-1})<0 ,
$$
and $(s_is_{i-1}bs_{i-1}s_i)(\alpha_{i+1})\neq\pm\alpha_{i+1}$ implies that $$(s_is_{i-1}bs_{i-1}s_i)s_{i+1}\neq s_{i+1}(s_is_{i-1}bs_{i-1}s_i).$$ It follows
from Corollary \ref{length0} that $$
s_{i+1}\ltimes (s_{i}\ltimes (s_{i-1}\ltimes b)=s_{i+1}s_is_{i-1}bs_{i-1}s_is_{i+1},\quad \rho(b)=\rho(s_{i+1}\ltimes s_i\ltimes s_{i-1}\ltimes b)+3.
$$
Set $d:=s_{i+1}\ltimes s_i\ltimes s_{i-1}\ltimes b$. Then $b=s_{i-1}\ltimes s_i\ltimes s_{i+1}\ltimes d$ and this is b)

By a similar reasoning (i.e., the same argument used in the proof of Lemma \ref{64a}), we can show that if $i<n-2$, then $b(i+3)<i+3$ implies that
$b=s_{i+2}\ltimes s_{i+1}\ltimes s_{i}\ltimes d$, where $d\in I_\ast$ and $\rho(b)=\rho(d)+3$, which is c) as required.

Therefore, we can assume that $b(i-1)\leq i-1$ whenever $i\geq 2$ and $b(i+3)\geq i+3$ whenever $i<n-2$. If $i=1$, then as $1<n-2$ we have that $b(4)\geq 4$.
It follows that $b(k)=k$ for any $k\geq 4$ and hence $b=1$ which contradicts our assumption. If $i\geq 2$, then as $s_0$ commutes with $s_i$ and $s_{i+1}$,
$s_0$ is not a descent of $b$, it follows that $b(\eps_1)>0$. Since $0<b(1)<b(2)<\cdots<b(i-1)\leq i-1$ and $i+3\leq b(i+3)<b(i+4)<\cdots<b(n)$, we conclude that $b(k)=k$ for
any $k$, which is again a contradiction. This completes the proof of the lemma.
\end{proof}

\begin{lem} \label{67aB} Let $W=W(B_n)$ be the Weyl group of type $B_n$. Let $w\in I_\ast$ be an involution, and let $s=s_{\alpha}$ and $t=s_{\beta}$ for some
$\alpha\neq\beta$ in $\Delta$ with $$(\alpha,\beta)\in\{(\alpha_i,\alpha_{i+1}),(\alpha_{i+1},\alpha_{i})|1\leq i<n-1\}.
$$
Assume that $s,t$ are both descents of $w$ and let $b\in I_\ast$ be the unique minimal length representative of $W_K wW_K$ where $K:=\<s,t\>$. Assume further that
$b$ has no descents which commute with both $s$ and $t$. Then $m_{st}=3$. Moreover, it holds that $bs=tb$ and $bt=sb$ (Case 6 in the notation of Lemma
\ref{7cases}) only if one of the following occurs: \begin{enumerate}
\item $b=1$;
\item $b=s_{i-1}\ltimes s_{i}\ltimes s_{i+1}\ltimes d$, where $2\leq i<n-1$, $d\in I_\ast$ and $\rho(b)=\rho(d)+3$;
\item $i=1$, $b=s_{3}\ltimes s_{2}\ltimes s_{1}\ltimes d$, where $d\in I_\ast$ and $\rho(b)=\rho(d)+3$;
\item $i=1$, $b=s_{0}\ltimes s_1\ltimes s_0 \ltimes s_2$;
\item $i=1$, $b=s_0\ltimes s_{3}\ltimes s_1\ltimes s_{2}\ltimes s_3 \ltimes s_1 \ltimes s_0 \ltimes s_1$.
\end{enumerate}
\end{lem}

\begin{proof} By assumption, we must have that $n\geq 4$. Suppose that a) does not happen, i.e., $b\neq 1$. By assumption, we have that $bsb=t$ and $btb=s$.
It follows that $b(\alpha)=\pm\beta, b(\beta)=\pm\alpha$ (by (\ref{reflection0B})).
By the expression of $w$ given in Case 4, both $(t, s, b)$ and $(s, t, b)$ are reduced $I_\ast$-sequences. Applying Corollary \ref{length0} and Corollary
\ref{deletion2}, we can deduce that $b(\alpha)>0<b(\beta)$. It follows that $b(\alpha)=\beta, b(\beta)=\alpha$.

Without loss of generality, we can assume that $(\alpha,\beta)=(\alpha_i,\alpha_{i+1})$ for some $1\leq i<n-1$. Then $b(i)=-(i+2), b(i+1)=-(i+1), b(i+2)=-i$.
By Lemma \ref{length1} and the assumption that $s_t$ is not a descent of $w$ for any $t<i-1$ or $t\geq i+3$, we can deduce that \begin{equation}\label{b01B2}
b(1)<b(2)<\dots<b(i-1)\,\,\,\text{and}\,\,\,b(i+3)<b(i+4)<\dots<b(n)
\end{equation}
and $b\neq 1$.
Suppose that a) does not happen, i.e., $b\neq 1$. There are only the following two possibilities:

\smallskip
{\it Case 1.} $(\alpha,\beta)=(\alpha_i,\alpha_{i+1})$ for some $2\leq i<n-1$. Since $s_0$ commutes with $s_i$ and $s_{i+1}$, we see that $s_0$ is not a
descent of $b$. Thus $b(\eps_1)>0$. It follows from (\ref{b01B2}) that $b(i-1)>0$. Then using the same argument in the proof of Lemma \ref{64aB}, we can show that
$b(i-1)>0$ implies that $b=s_{i-1}\ltimes s_i\ltimes s_{i+1}\ltimes d$ with $\rho(b)=\rho(d)+3$, which is b) as required.

\smallskip
{\it Case 2.} $(\alpha,\beta)=(\alpha_1,\alpha_2)$. Then $b(1)=-3, b(3)=-1, b(2)=-2$. By (\ref{b01B2}), we have that $b(4)<b(5)<\dots<b(n)$.

If $b(4)=\pm4$, then by (\ref{b01B2})  we can deduce that $b(j)=j$ for any $j\geq5$. It follows that either $$
b=(s_1s_2s_1)(s_0)(s_1s_0s_1)(s_2s_1s_0s_1s_2)=s_0 s_1 s_0s_2 s_1 s_0=s_{0}\ltimes s_1\ltimes s_0 \ltimes s_2,
$$
or $$\begin{aligned}
b&=(s_1s_2s_1)(s_0)(s_1s_0s_1)(s_2s_1s_0s_1s_2)(s_3s_2s_1s_0s_1s_2s_3)\\
&=s_0s_3s_1s_2s_1s_0s_1s_0s_3s_2s_1s_3s_0=s_0\ltimes s_{3}\ltimes s_1\ltimes s_{2}\ltimes s_3 \ltimes s_1 \ltimes s_0 \ltimes s_1,
\end{aligned}$$
which are d) and e) respectively as required.

It remains to consider the case when $b(4)\leq -5$. In this case, using the same argument in the proof of Lemma \ref{67a}, we can show that
$b(4)\leq -5$ implies that $b=s_{3}\ltimes s_2\ltimes s_{1}\ltimes d$ with $\rho(b)=\rho(d)+3$, which is c) as required. This completes the proof of the Lemma.
\end{proof}

\begin{rem} \label{rem0B} We consider reduce $I_\ast$-expressions for involutions in the Weyl group of type $B_n$. In this case, in addition to the basic braid
$I_\ast$-transformations given by Proposition \ref{braid}, one clearly has to add the following natural ``right end transformations": $$
s_i\ltimes s_{i+1}\longleftrightarrow s_{i+1}\ltimes s_i,\,\,\,s_0\ltimes s_1\ltimes s_{0}\longleftrightarrow s_1\ltimes s_{0}\ltimes s_1,\,\,
s_k\ltimes s_{j}\longleftrightarrow s_{j}\ltimes s_k,
$$
where $1\leq i<n-1$, $0\leq j,k<n$, $|j-k|>1$. However, as in the type $D_n$ case, these are {\it NOT} all the basic braid $I_\ast$-transformations for
involutions in $W(B_n)$ that we need. In fact, one has to add an extra transformation in the case of type $B_n$ (see the last transformation in
Definition \ref{braid1B}), which is a new phenomenon for type $B_n$.
\end{rem}

\begin{dfn}\label{braid1B} By a basic braid $I_\ast$-transformation, we mean one of the following transformations and their inverses: $$
\begin{aligned}
1)\quad &(\cdots,s_0,s_{1},s_0,s_1\cdots)\longmapsto(\cdots,s_{1},s_0,s_{1},s_0,\cdots) ,\\
2)\quad &(\cdots,s_j,s_{j+1},s_j,\cdots)\longmapsto(\cdots,s_{j+1},s_j,s_{j+1},\cdots) ,\\
3)\quad &(\cdots,s_b,s_{c},\cdots)\longmapsto (\cdots,s_{c},s_b,\cdots) ,\\
4)\quad & (\cdots,s_k,s_{k+1})\longmapsto (\cdots,s_{k+1},s_k) ,\\
5)\quad & (\cdots,s_0,s_1,s_0)\longmapsto (\cdots,s_1,s_0,s_1) ,\\
6)\quad & (\cdots,s_0,s_1,s_0,s_2,s_1,s_0)\longmapsto (\cdots,s_1,s_0,s_1,s_2,s_1,s_0) ,
\end{aligned}
$$
where all the sequences appearing above are reduced sequences, and the entries marked by corresponding ``$\cdots$" must match, and in the first two
transformations (i.e., 1) and 2)) we further require that the right end part entries marked by ``$\cdots$" must be non-empty.
We define a braid $I_\ast$-transformation to be the composition of a series of basic braid $I_\ast$-transformations.
\end{dfn}

\begin{thm} \label{braid0B} Let $(s_{i_1},\cdots,s_{i_k}), (s_{j_1},\cdots,s_{j_k})$ be two reduced $I_\ast$-sequences which can be transformed into each other through a series of basic braid transformations. Then $$
s_{i_1}\ltimes s_{i_2}\ltimes\dots\ltimes s_{i_k}=s_{j_1}\ltimes s_{j_2}\ltimes\dots\ltimes s_{j_k} .
$$
\end{thm}
\begin{proof} This follows easily from Proposition \ref{braid}, Definition \ref{braid1B} except for the last transformation 6) in Definition \ref{braid1B}. In fact,
by direct calculation, we can get that $$\begin{aligned}
s_0\ltimes s_1\ltimes s_0\ltimes s_2\ltimes s_1\ltimes s_0&=s_0s_1s_0s_2s_1s_0s_1s_2s_1=s_1s_0s_2s_1s_0s_1s_2s_1s_0\\
&=s_1\ltimes s_0\ltimes s_1\ltimes s_2\ltimes s_1\ltimes s_0 \,\,,
\end{aligned}$$
as required.
\end{proof}

As in Section 3, we sometimes use the simplified notations $(i_1,\cdots,i_k)\longleftrightarrow (j_1,\cdots,j_k)$ and $(s_{i_1},\cdots,s_{i_k})\longleftrightarrow
(s_{j_1},\cdots,s_{j_k})$ in place of  $s_{i_1}\ltimes \cdots\ltimes s_{i_k}\longleftrightarrow s_{j_1}\ltimes \cdots\ltimes s_{j_k}$.

\begin{thm} \label{mainthm0B} Let $w\in I_\ast$. Then any two reduced $I_\ast$-expressions for $w$ can be transformed into each other through a series of basic braid $I_\ast$-transformations.
\end{thm}

\begin{proof} We prove the theorem by induction on $\rho(w)$. Suppose that the theorem holds for any $w\in I_\ast$  with $\rho(w)\leq k$. Let $w\in I_\ast$ with
 $\rho(w)=k+1$. Let $(s_{i_0},s_{i_1},s_{i_2},\cdots,s_{i_k})$ and $(s_{j_0},s_{j_1},s_{j_2},\cdots,s_{j_k})$ be two reduced $I_\ast$-expressions for
  $w\in I_\ast$. We need to prove that \begin{equation}\label{goalB}
(i_0,i_1,\cdots,i_k)\longleftrightarrow (j_0,j_1,\cdots,j_k).
\end{equation}
For simplicity, we set $s=s_{i_0}, t=s_{j_0}$. Let $m$ be the order of $st$.

As in the proof of Theorem \ref{mainthm0}, we shall use the strategy (\ref{strategy1}). That says, we want to show that there exists some $s_\alpha\in S$ such
that $(i_0,\dots,i_k)\longleftrightarrow (s_\alpha,\dots)$ and $(j_0,\dots,j_k)\longleftrightarrow (s_\alpha,\dots)$. Once this is proved, then by induction
hypothesis we are done.

Suppose that $m=3$. Then we are in the situations of Cases 1,2,3,4,6 of Lemma \ref{7cases}. Suppose that we are in Cases 1,2,3 of Lemma \ref{7cases}. Then we get
that \begin{equation}\label{123}
s_{i_0}\ltimes s_{i_1}\ltimes\cdots\ltimes s_{i_k}\longleftrightarrow s\ltimes t\ltimes s\ltimes b\longleftrightarrow t\ltimes s\ltimes t\ltimes
b\longleftrightarrow s_{j_0}\ltimes s_{j_1}\ltimes\cdots\ltimes s_{j_k},
\end{equation}
where the first and the third ``$\longleftrightarrow$" follows from induction hypothesis, and the second ``$\longleftrightarrow$" follows from the expression of
$w$ given in Cases 1,2,3 of Lemma \ref{7cases}. It remains to consider Cases 4,6 of Lemma \ref{7cases}. To this end, we shall apply Lemmas \ref{64aB} and
\ref{67aB}. In these cases, if $b=1$, then $k=1$ and $s_{i_0}\ltimes s_{i_1}=s\ltimes t\longleftrightarrow t \ltimes s=s_{j_0}\ltimes s_{j_1}$.
Henceforth, we assume that $b\neq 1$.

Suppose that we are in Case b) of Lemma \ref{64aB} or of Lemma \ref{67aB}. Without loss of generality, we assume that $i_0=i$ and $j_0=i+1$. Then $$\begin{aligned}
&\quad (i_0,i_1,\cdots,i_k)\longleftrightarrow (i,i+1,\underbrace{i-1,i,i+1,d})\longleftrightarrow (i,\underbrace{i-1,i+1},i,i+1,d)
\longleftrightarrow\\
&\quad (i,i-1,\underbrace{i,i+1,i},d)\longleftrightarrow (\underbrace{i-1,i,i-1},i+1,i,d)
\longleftrightarrow (i-1,\underbrace{i+1,i,i-1,i+1,d})\\
&\quad \longleftrightarrow (\underbrace{i+1,i-1},i,i-1,i+1,d)\longleftrightarrow
(i+1,\underbrace{i,i-1,i},i+1,d)\\
&\qquad \longleftrightarrow (j_0,j_1,\cdots,j_k),
\end{aligned}
$$
as required.

Suppose that we are in Case c) of Lemma \ref{64aB}. Without loss of generality, we assume that $i_0=i$ and $j_0=i+1$. Then $$\begin{aligned}
&\quad (i_0,i_1,\cdots,i_k)\longleftrightarrow (i,i+1,\underbrace{i+2,i+1,i,d})\longleftrightarrow (i,\underbrace{i+2,i+1,i+2},i,d)
\longleftrightarrow\\
&\quad (\underbrace{i+2,i},i+1,i+2,i,d)\longleftrightarrow (i+2,\underbrace{i+1,i+2,i,i+1,d})
\longleftrightarrow\\
&\quad  (\underbrace{i+1,i+2,i+1},i,i+1,d)\longleftrightarrow (i+1,i+2,\underbrace{i,i+1,i},d)\longleftrightarrow\\
&\qquad (i+1,\underbrace{i,i+2},i+1,i,d)\longleftrightarrow (j_0,j_1,\cdots,j_k),
\end{aligned}
$$
as required. A similar argument also applies if we are in Case c) of Lemma \ref{67aB}.

Suppose that we are in Case d) of Lemma \ref{67aB}. Without loss of generality, we assume that $i_0=1$ and $j_0=2$. Then $$\begin{aligned}
&(i_0,i_1,\cdots,i_k)\longleftrightarrow (1, 2, 0, 1, 0, 2)\longleftrightarrow (1, \underbrace{0, 2}, 1, \underbrace{2, 0})\longleftrightarrow (1, 0,
\underbrace{1, 2, 1}, 0)\\
&\qquad\longleftrightarrow (\underbrace{0, 1, 0, 2, 1, 0})\longleftrightarrow (0, 1, \underbrace{2, 0}, 1, 0)\longleftrightarrow(0, 1, 2, \underbrace{1, 0, 1})
\longleftrightarrow (0, \underbrace{2, 1, 2}, 0, 1)\\
&\qquad \longleftrightarrow (\underbrace{2, 0}, 1, \underbrace{0, 1, 2})\longleftrightarrow (2, 1, 0, 1, 0, 2)
\longleftrightarrow (j_0,j_1,\cdots,j_k),
\end{aligned}
$$
as required.

Suppose that we are in Case e) of Lemma \ref{67aB}. Without loss of generality, we assume that $i_0=1$ and $j_0=2$. Then  $$\begin{aligned}
&\quad\,\,(i_0,i_1,\cdots,i_k)\longleftrightarrow (1, 2, \underbrace{0, 3, 1, 2, 3, 1, 0, 1})\\
&\longleftrightarrow
(1, 2, \underbrace{3, 0}, 1, 2, 3, 1, 0, 1)\longleftrightarrow
(1, 2, 3, 0, 1, 2, \underbrace{1, 0, 1, 3})\\
&\longleftrightarrow
(1, 2, 3, 0, \underbrace{2, 1, 2}, 0, 1, 3)\longleftrightarrow
(1, 2, 3, \underbrace{2, 0}, 1, 2, \underbrace{3, 0, 1})\\
&\longleftrightarrow
(1, \underbrace{3, 2, 3}, 0, 1, 2, 3, 0, 1)\longleftrightarrow
(\underbrace{3, 1}, \underbrace{0, 2, 1, 3}, 2, 3, 0, 1)\\
&\longleftrightarrow
(3, 1, 0, 2, 1, \underbrace{2, 3, 2}, 0, 1)\longleftrightarrow
(3, 1, \underbrace{2, 0}, 1, \underbrace{0, 2, 3, 2}, 1)\\
&\longleftrightarrow
(3, 1, 2, 0, 1, 0, 2, \underbrace{1, 2, 3})\longleftrightarrow
(3, 1, 2, 0, 1, 0, \underbrace{1, 2, 1}, 3)\\
&\longleftrightarrow
(3, 1, 2, \underbrace{1, 0, 1, 0}, 2, 1, 3)\longleftrightarrow
(3, \underbrace{2, 1, 2}, 0, 1, 0, 2, 1, 3)\\
&\longleftrightarrow
(3, 2, 1, \underbrace{0, 2}, 1, \underbrace{2, 0}, 1, 3)\longleftrightarrow
(3, 2, 1, 0, \underbrace{1, 2, 1}, 0, 1, 3)\\
&\longleftrightarrow
(3, 2, 1, 0, 1, 2, \underbrace{3, 1, 0, 1})\longleftrightarrow
(3, 2, 1, 0, 1, 2, 3, \underbrace{0, 1, 0})\\
&\longleftrightarrow
(3, 2, 1, 0, 1, \underbrace{0, 2, 3}, 1, 0)\longleftrightarrow
(3, 2, \underbrace{0, 1, 0, 1}, 2, 3, 1, 0)\\
&\longleftrightarrow
(\underbrace{0, 3, 2},  1, 0, 1, 2, 3, 1, 0).
\end{aligned}
$$
On the other hand, according to the definition of the braid $I_\ast$-transformations, we get that $$\begin{aligned}
&\quad\,(j_0,j_1,\cdots,j_k)\longleftrightarrow (2, 1, 0, 3, 1, 2, 3, 1, 0, 1 )\\
&\longleftrightarrow
(2, 1, 0, \underbrace{1, 3}, 2, 3, 0, 1, 0)\longleftrightarrow
(2, 1, 0, 1, \underbrace{0, 3, 2, 3}, 1, 0)\\
&\longleftrightarrow
(2, \underbrace{0, 1, 0, 1}, 3, 2, 3, 1, 0)\longleftrightarrow
(\underbrace{0, 2}, 1, 0, 1, 3, 2, 3, 1, 0).\\
\end{aligned}$$
So again we are in a position to apply the induction hypothesis. This completes the proof of (\ref{goalB}) when $m=3$.

Now suppose that $m=4$. Then by Lemma \ref{64bB} we are in the situations of Cases 1,2,3,5 of Lemma \ref{7cases}. Suppose that either $b=1$ or we are in Case
1,2,3 of Lemma
\ref{7cases}. Then we
can use a similar argument as (\ref{123}) to show that (\ref{goalB}) holds. It remains to consider Case 5 of Lemma \ref{7cases}. To this end, we shall apply
Lemma \ref{64bB}.

Suppose that we are in Case b) of Lemma \ref{64bB}. Without loss of generality, we assume that $i_0=0$ and $j_0=1$. Then using the 6th transformation given in
Definition \ref{braid1B}, we can get that $$
(i_0,i_1,\cdots,i_k)\longleftrightarrow (0,1,0,2,1,0)\longleftrightarrow (1,0,1,2,1,0)\longleftrightarrow (j_0,j_1,\cdots,j_k),
$$
as required.

Suppose that we are in Case c) of Lemma \ref{64bB}. Without loss of generality, we assume that $i_0=0$ and $j_0=1$.
$$\begin{aligned}
&\quad\,\,(i_0,i_1,\cdots,i_k)\longleftrightarrow (0, 1, 0, 2, 1, 0, 1, 2, d)\longleftrightarrow (0, 1, \underbrace{2, 0}, 1, 0, 1, 2, d)\\
&\longleftrightarrow (0, 1, 2, \underbrace{1, 0, 1, 0} 2, d)\longleftrightarrow (0, \underbrace{2, 1, 2}, 0, 1, 0, 2, d)\\
&\longleftrightarrow (\underbrace{2, 0}, 1, 2, 0, 1, 0, 2, d),
\end{aligned}$$
On the other hand, $$\begin{aligned}
&\quad\,\,(j_0,j_1,\cdots,j_k)\longleftrightarrow (1, 0, 1, 2, 1, 0, 1, 2, d)\longleftrightarrow (1, 0, \underbrace{2, 1, 2}, 0, 1, 2, d)\\
&\longleftrightarrow (1, \underbrace{2, 0}, 1, \underbrace{0, 2}, 1, 2, d)\longleftrightarrow (1, 2, 0, 1, 0, \underbrace{1, 2, 1}, d)\\
&\longleftrightarrow  (1, 2, \underbrace{1, 0, 1, 0}, 2, 1, d)\longleftrightarrow  (\underbrace{2, 1, 2}, 0, 1, 0, 2, 1, d) .
\end{aligned}$$
Therefore, we can apply induction hypothesis to get that $(i_0,i_1,\cdots,i_k)\longleftrightarrow (j_0,j_1,\cdots,j_k)$. This completes the proof of
(\ref{goalB}) when $m=4$.

We make the following useful observation.

{\it Observation 2.} If there exist $s_\alpha, s_\beta\in S$ such that $s_\alpha s_\beta$ has order $1, 3$ or $4$, and $$
(i_0,i_1,\dots,i_k)\longleftrightarrow
(s_\alpha,\dots)\,\,\,\text{and}\,\,\,(j_0,j_1,\dots,j_k)\longleftrightarrow (s_\beta,\dots),
$$ then by the results we have obtained we can deduce that $(i_0,i_1,\dots,i_k)\longleftrightarrow
(j_0,j_1,\dots,j_k)$.

It remains to consider the case when $m=2$. That says, $st=ts$. We use a similar argument as in the proof of Theorem \ref{mainthm0} (cf. the paragraphs after
Observation 1). As in the proof of Theorem \ref{mainthm0}, there exists some integer $0\leq a\leq k$, such that $$
s_{j_0}\ltimes(s_{i_0}\ltimes s_{i_1}\ltimes s_{i_2}\ltimes\cdots\ltimes s_{i_k})=s_{i_0}\ltimes s_{i_1}\ltimes s_{i_2}\ltimes\cdots
\ltimes s_{i_{a-1}}\ltimes s_{i_{a+1}}\ltimes\cdots\ltimes s_{i_k} .
$$
In order to prove (\ref{goalB}), it suffices to show that \begin{equation}\label{goal2B0} (j_0,i_0,i_1,i_2,\cdots,i_{a-1},i_{a+1},\cdots,i_k)\longleftrightarrow
(i_0,i_1,i_2,\cdots,i_k).\end{equation}
Moreover, just as in the proof of Theorem \ref{mainthm0}, using Observation 2 we can consider only the case when $a=0$ and $j_0=i_0+2$. So in this case,
(\ref{goal2B0}) is reduced
to \begin{equation}\label{goal2B} (j_0,i_1,i_2,\cdots,i_k)\longleftrightarrow
(i_0,i_1,i_2,\cdots,i_k).\end{equation}

We write $w_1=s_{i_1}\ltimes s_{i_2}\ltimes\cdots\ltimes s_{i_k}$. If $i_1<i_0-1$ or $i_1>i_0+3$, then $$
(i_0,i_1,\cdots)\longleftrightarrow (i_1,i_0,\cdots)\longleftrightarrow (i_1,j_0,\cdots)\longleftrightarrow (j_0,i_1,\cdots),
$$
as required. If $i_1=i_0-1$, then by Observation 2, $$
(j_0,i_1,\cdots)\longleftrightarrow (j_0,i_0-1,\cdots)\longleftrightarrow (i_0-1,j_0,\cdots)\longleftrightarrow (i_0,i_1,\cdots),
$$
as required. A similar argument applies to the case when $i_1=i_0+3$. Therefore, we can consider only the case when $i_1=i_0+1$. Since $s_{i_0}\ltimes s_{i_0+1}
\neq s_{i_0+2}\ltimes s_{i_0+1}$, we get that $w_2:=s_{i_2}\ltimes\cdots\ltimes s_{i_k}\neq 1$.

Note that $i_2\neq i_0+1$ because $s_{i_1}\ltimes s_{i_2}\ltimes\cdots\ltimes s_{i_k}$ is reduced. If $i_2<i_0-1$ or $i_2>i_0+3$, then $$
(i_0,i_1,i_2,\cdots)\longleftrightarrow (i_2,i_0,i_1,\cdots)\longleftrightarrow (i_2,j_0,i_1,\cdots)\longleftrightarrow (j_0,i_1,i_2,\cdots),
$$
as required. If $i_2=i_0-1$, then by Observation 2, $$
(j_0,i_1,i_2,\cdots)\longleftrightarrow (i_2,j_0,i_1,\cdots)\longleftrightarrow (i_0-1,j_0,\cdots)\longleftrightarrow (i_0,i_1,i_2,\cdots),
$$
as required. A similar argument applies to the case when $i_2=i_0+3$. If $i_2=i_0+2$, then by Observation 2, $$
(j_0,i_1,i_2,\cdots)\longleftrightarrow (i_0+2,i_0+1,i_0+2,\cdots)\longleftrightarrow (i_0+1,i_0,i_0+1,\cdots)\longleftrightarrow (i_0,i_1,i_2,\cdots),
$$
as required. A similar argument applies to the case when $i_0\geq 1$ and $i_2=i_0$. Therefore, we can consider only the case when $i_2=i_0=0$.
Since $s_{0}\ltimes s_{1}\ltimes s_0\neq s_{2}\ltimes s_{1}\ltimes s_0$, we get that $w_3:=s_{i_3}\ltimes\cdots\ltimes s_{i_k}\neq 1$.

Now $i_3\neq 0$ because $s_{i_2}\ltimes s_{i_3}\ltimes\cdots\ltimes s_{i_k}$ is reduced. If $i_3>3$, then $$\begin{aligned}
(i_0,i_1,i_2,i_3,\cdots)&\longleftrightarrow (0,1,0,i_3,\cdots)\longleftrightarrow (i_3,0,1,0,\cdots)\longleftrightarrow
(i_3,2,1,0,\cdots)\\
&\quad \longleftrightarrow (j_0,i_1,i_2,i_3,\cdots),
\end{aligned} $$
as required. If $i_3=3$, then by Observation 2,
$$\begin{aligned}
(i_0,i_1,i_2,i_3,\cdots)&\longleftrightarrow (0,1,0,3,\cdots)\longleftrightarrow (3,0,1,0,\cdots)\longleftrightarrow
(2,1,0,1,\cdots)\\
&\quad \longleftrightarrow (j_0,i_1,i_2,i_3,\cdots),
\end{aligned} $$
as required.  If $i_3=2$, then by Observation 2,
$$\begin{aligned}
(j_0,i_1,i_2,i_3,\cdots)&\longleftrightarrow (2,1,0,2,\cdots)\longleftrightarrow (2,1,2,0,\cdots)\longleftrightarrow
(1,2,1,0,\cdots)\\
&\quad \longleftrightarrow (0,1,0,2,\cdots)\longleftrightarrow (i_0,i_1,i_2,i_3,\cdots),
\end{aligned} $$
as required. Finally, if $i_3=1$, then as $s_0\ltimes s_1\ltimes s_0\ltimes s_1$ is not reduced, we can deduce that
$w_4:=s_{i_4}\ltimes\cdots\ltimes s_{i_k}\neq 1$. In this case, by Observation 2, $$\begin{aligned}
(i_0,i_1,i_2,i_3,\cdots)&\longleftrightarrow (0,1,0,1,\cdots)\longleftrightarrow (1,0,1,0,\cdots)\longleftrightarrow
(2,1,0,1,\cdots)\\
&\quad \longleftrightarrow (j_0,i_1,i_2,i_3,\cdots),
\end{aligned} $$
as required. This completes the proof of (\ref{goal2B}) and hence the proof of (\ref{goalB}) when $m=2$. This finishes the proof of the theorem.
\end{proof}

\bigskip

\section{An application}

The basic braid $I_{\ast}$-transformations which we found in Definition \ref{braid1} and \ref{braid1B} and Theorem \ref{mainthm0} and \ref{mainthm0B} are
 very useful for analysing the Hecke module structures on the space spanned by involutions. The point is that it reduces the verification of Hecke defining
 relations to a {\it finite} doable calculations. In this section, we shall give an application of this observation to Lusztig's conjecture which is our
 original motivation.

\begin{lem} \label{decomp1} The elements in the following set $$\begin{aligned}
&\bigl\{(s_{c_1}s_{c_1-1}\dots s_2 s_1)(s_{c_2}s_{c_2-1}\dots s_2 s_u)\dots (s_{c_{k-1}}s_{c_{k-1}-1}\dots s_2 s_1)(s_{c_k}s_{c_k-1}\dots s_2 s_u)\\
&\qquad \bigm|\text{$k$ is even, and $1\leq c_1<c_2<\dots<c_k<n$}\bigr\}\bigcup\\
&\bigl\{(s_{c_1}s_{c_1-1}\dots s_2 s_u)(s_{c_2}s_{c_2-1}\dots s_2 s_1)\dots (s_{c_{k-1}}s_{c_{k-1}-1}\dots s_2 s_1)(s_{c_k}s_{c_k-1}\dots s_2 s_u)\\
&\qquad \bigm|\text{$k$ is odd, and $1\leq c_1<c_2<\dots<c_k<n$}\bigr\}
\end{aligned}
$$
is a complete set of left coset representatives of $\langle s_1,s_2,\dots,s_{n-1}\rangle$ in $W(D_n)$\footnote{By convention, if $c_1=1$ then
$s_{c_1}s_{c_1-1}\dots s_2 s_u:=s_u$.}.
\end{lem}

\begin{proof} One can check directly that the elements in the above two sets are minimal length left coset representatives of
$\langle s_1,s_2,\dots,s_{n-1}\rangle$ in $W(D_n)$. Moreover, it is a complete set by a counting argument (using the well-known fact that
$|W(D_n)|=2^{n-1}|\Sym_n|$).
\end{proof}

By \cite{LV1}, there is an $\mathcal{H}_u$-module structure on $M$ which is defined as follows: for any $s\in S$ and any $w\in I_\ast$, $$\begin{aligned}
& T_s a_w=ua_w+(u+1)a_{sw}\quad\text{if\, $sw=ws>w$;}\\
& T_s a_w=(u^2-u-1)a_w+(u^2-u)a_{sw}\quad\text{if\, $sw=ws<w$;}\\
& T_s a_w=a_{sws}\quad\text{if\, $sw\neq ws>w$;}\\
& T_s a_w=(u^2-1)a_w+u^2a_{sws}\quad\text{if\, $sw\neq ws<w$.}\\
\end{aligned}
$$

\begin{lem} \label{keylem1} Let $W\in\{W(B_n),W(D_n)\}$ and $\ast=\text{id}$. The map $a_1\mapsto X_{\emptyset}$ can be extended to a well-defined
$\Q(u)$-linear map $\eta_0$
from $\Q(u)\otimes_{\mathcal{A}}M$ to $(\Q(u)\otimes_{\mathcal{A}}\mathcal{H}_u)X_{\emptyset}$ such that for any $w\in I_{\ast}$ and any reduced
$I_\ast$-expression $\sigma=(s_{j_1},\cdots,s_{j_k})$ for $w$, $$
\eta_0(a_w)=\theta_{\sigma}(X_{\emptyset}):=\theta_{\sigma,1}\circ\theta_{\sigma,2}\circ\cdots\circ\theta_{\sigma,k}(X_{\emptyset}),
$$
where for each $1\leq t\leq k$, if $$s_{j_t}(s_{j_{t+1}}\ltimes s_{j_{t+2}}\ltimes\cdots\ltimes s_{j_k})\neq (s_{j_{t+1}}\ltimes s_{j_{t+2}}\ltimes\cdots\ltimes
s_{j_k})s_{j_t}>(s_{j_{t+1}}\ltimes s_{j_{t+2}}\ltimes\cdots\ltimes s_{j_k}),$$ then we define $\theta_{\sigma,t}:=T_{s_{j_t}}$; while if $$
s_{j_t}(s_{j_{t+1}}\ltimes s_{j_{t+2}}\ltimes\cdots\ltimes s_{j_k})=(s_{j_{t+1}}\ltimes s_{j_{t+2}}\ltimes\cdots\ltimes s_{j_k})s_{j_t}>(s_{j_{t+1}}\ltimes
s_{j_{t+2}}\ltimes\cdots\ltimes s_{j_k}),$$
then we define $\theta_{\sigma,t}:=(T_{s_{j_t}}-u)/(u+1)$.
\end{lem}

\begin{proof} We only prove the lemma for the case when $W=W(D_n)$ as the other case is similar. The idea of the proof is essentially the same as we used in
the proof of \cite[Lemma 5.1]{HZH1} except that we now use Proposition \ref{braid} and have
to verify one more relation: $$\begin{aligned}
&\frac{T_2-u}{u+1}T_3T_uT_1T_2\frac{T_u-u}{u+1}\frac{T_1-u}{u+1}\frac{T_3-u}{u+1}X_{\emptyset}=\\
&\qquad\qquad \frac{T_3-u}{u+1}T_2T_uT_1T_2\frac{T_u-u}{u+1}\frac{T_1-u}{u+1}\frac{T_3-u}{u+1}X_{\emptyset} .
\end{aligned}$$

To check this relation, we can assume without loss of generality that $n=4$ because $\sum_{w\in\mathcal{H}_u(D_4)}u^{-\ell(w)}T_w$ is a left factor of
$X_{\emptyset}$.

To simplify the notation, we set $$
T_u\dots T_u:=T_uT_2T_1T_3T_2T_u,\quad T_2\dots T_u:=T_2T_1T_3T_2T_u,\quad Y_{\emptyset}:=\sum_{w\in\Sym_4}u^{-\ell(w)}T_w .
$$
By direct calculation, one can get that $$\begin{aligned}
&\quad\,T_uT_1T_2(T_3-u)(T_1-u)(T_u-u)X_{\emptyset}\\
&=T_uT_1T_2(T_3-u)(T_1-u)(T_u-u)(1+u^{-1}T_u+u^{-2}T_2T_u+u^{-3}T_3T_2T_u\\
&\qquad +u^{-3}T_1T_2T_u+u^{-4}T_1T_3T_2T_u+u^{-5}T_2T_1T_3T_2T_u+u^{-6}T_uT_2T_1T_3T_2T_u)
Y_{\emptyset}\\
&=\Bigl(uT_1T_2T_uT_2-T_1T_2T_uT_2T_1-T_1T_2T_uT_2T_3+u^{-1}T_1T_2T_uT_2T_1T_3+u^{-2}T_2\dots T_uT_3T_2\\
&\qquad -u^{-3}T_2\dots T_uT_3T_2T_1
+u^{-4}T_2\dots T_uT_3T_2T_1T_3-u^{-3}T_2\dots T_uT_3T_2T_3+\\
&\qquad 2(u^{-2}-u^{-4})T_u\dots T_uT_3T_2-
(u^{-3}-u^{-5})T_u\dots T_uT_3T_2T_1-(u^{-1}-u^{-3})T_u\dots T_uT_3\\
&\qquad +(u^{-4}-u^{-6})T_u\dots T_uT_3T_2T_1T_3-
(u^{-3}-u^{-5})T_u\dots T_uT_3T_2T_3\Bigr)Y_{\emptyset} .
\end{aligned}
$$

Let $Z_0$ be the element in the big bracket of the last equality. We want to show that $(T_2T_3-T_3T_2+uT_2-uT_3)Z_0Y_{\emptyset}=0$. Using Lemma \ref{decomp1},
 we see that if we express $(T_2T_3-T_3T_2+uT_2-uT_3)Z_0$ as a linear combination of standard bases $\{T_w|w\in W(D_n)\}$, then $T_w$ occurs with non-zero
  coefficient only if $$
w\in s_u\dots s_u\Sym_4\bigcup s_2\dots s_u\Sym_4\bigcup s_3s_1s_2s_u\Sym_4 \bigcup s_1s_2s_u\Sym_4 .
$$

By the proof of \cite[Lemma 5.1]{HZH1}, we know that for any $1\leq j<3$, \begin{equation}
\label{kill}
T_j(T_{j+1}-u)Y_{\emptyset}=T_{j+1}(T_{j}-u)Y_{\emptyset} .
\end{equation}
Using (\ref{kill}) and consider the above four left cosets separately, one can check that $(T_2T_3-T_3T_2+uT_2-uT_3)Z_0Y_{\emptyset}=0$. For example, if we
consider the term $T_w$ which occurs in $(T_2T_3-T_3T_2+uT_2-uT_3)Z_0Y_{\emptyset}$ such that $w\in s_1s_2s_u\Sym_4$, then we shall get that it is equal to
$$\begin{aligned}
&\quad\,\bigl(u^2T_1T_2T_uT_1T_2-uT_1T_2T_uT_1T_2T_1-uT_1T_2T_uT_1T_2T_3+T_1T_2T_uT_1T_2T_1T_3\\
&\qquad -u^2T_1T_2T_uT_3T_2+uT_1T_2T_uT_3T_2T_3+uT_1T_2T_uT_3T_2T_1-T_1T_2T_uT_3T_2T_1T_3\bigr)Y_{\emptyset}\\
&=T_1T_2T_u\bigl(u^2T_1T_2-uT_1T_2T_1-uT_1T_2T_3-u^2T_3T_2+uT_3T_2T_3+uT_3T_2T_1\\
&\qquad +T_2T_1T_2T_3-T_2T_3T_2T_1\bigr)Y_{\emptyset}\\
&=T_1T_2T_u\biggl(u^2T_1T_2-uT_1T_2T_1-uT_1(T_3T_2-uT_3+uT_2)-u^2T_3T_2+uT_3T_2T_3+\\
&\qquad uT_3(T_1T_2-uT_1+uT_2)+T_2T_1(T_3T_2-uT_3+uT_2)-T_2T_3(T_1T_2-uT_1+uT_2)\biggr)Y_{\emptyset}\\
&=0 ,
\end{aligned}
$$
as required, where we have used (\ref{kill}) in the second equality. For the other three cosets, one can do some similar calculation. We leave the details to
the readers.
\end{proof}

\begin{cor}\label{keycor4} With the notations as in Lemma \ref{keylem1}, the $\Q(u)$-linear map $\eta_0$ is a
left $(\Q(u)\otimes_{\mathcal{A}}\mathcal{H}_u)$-module homomorphism. In particular, $\eta_0=\eta$ is a well-defined surjective left
$(\Q(u)\otimes_{\mathcal{A}}\mathcal{H}_u)$-module homomorphism from $\Q(u)\otimes_{\mathcal{A}}M$ onto $(\Q(u)\otimes_{\mathcal{A}}\mathcal{H}_u)X_{\emptyset}$.
\end{cor}

\begin{proof} This follows from Lemma \ref{keylem1} and a similar argument used in the proof \cite[Lemma 5.3, Theorem 5.5]{HZH1}.
\end{proof}

\begin{rem} 1) Lusztig proved in \cite{Lu3} that his conjecture holds for any Coxeter group and any $\ast$ by completely different approach. In other words,
$\eta$ is always a left $(\Q(u)\otimes_{\mathcal{A}}\mathcal{H}_u)$-module isomorphism from $\Q(u)\otimes_{\mathcal{A}}M$ onto
$(\Q(u)\otimes_{\mathcal{A}}\mathcal{H}_u)X_{\emptyset}$.

2) We conjecture that for any Coxeter system $(W,S)$ with $S$ finite and any automorphism ``$\ast$" as described in Section 1, there exists a finite
set of basic braid $I_\ast$-transformations that span and preserve the reduced $I_\ast$-expressions for any twisted involutions in $W$ and which can be explicitly
described for all $W$ simultaneously.
\end{rem}


\bigskip


\begin{thebibliography}{00}


\bibitem{BB}
{\sc A.~Bj\"{o}rner, F.~Brenti}, {\em Combinatorics of Coxeter groups}, Graduate Texts in Mathematics, {\bf 231}, Springer, 2005.

\bibitem{CJ}
{\sc M.B.~Can, M.~Joyce}, {\em Weak order on complete quadrics}, Trans. Amer. Math. Soc., {\bf 365}(12) (2013), 6269--6282.

\bibitem{CJW}
{\sc M.B.~Can, M.~Joyce, B.~Wyser}, {\em Chains in weak order posets associated to involutions},
J. Combin. Theory Ser. A, {\bf 137} (2016), 207--225.

\bibitem{DDPW}
{\sc B.~Deng, J.~Du, B.~Parshall and J.~Wang}, {\em Finite dimensional algebras and quantum groups}, Mathematical Surveys and Monographs, {\bf 150}, American Mathematical Society, 2008.

\bibitem{Gec}
{\sc M.~Geck}, {\em Hecke algberas of finite type are cellular},  Invent. Math., {\bf 169} (2007), 501--517.

\bibitem{GP}
{\sc M.~Geck and G. Pfeiffer}, {\em Characters of finite Coxeter groups and Iwahori-Hecke algebras}, London Mathematical Society Monographs New Series
{\bf 446}, Clarendon Press Oxford, 2000.

\bibitem{HMP1}
{\sc Z.~Hamaker, E.~Marberg and B.~Pawlowski}, {\em Involution words: counting problems and connections to Schubert calculus for symmetric orbit closures}, preprint, arXiv: arXiv:1508.01823, 2015.

\bibitem{HMP2}
\leavevmode\vrule height 2pt depth -1.6pt width 23pt, {\em Involution words II: braid relations and atomic structures}, preprint, arXiv:1601.02269, 2016.

\bibitem{HuJ2} {\sc J.~Hu}, {\em Quasi-parabolic subgroups of the Weyl group of
type $D$}, European Journal of Combinatorics, {\bf 28} (2007), 807--821.

\bibitem{Hu1}
{\sc A.~Hultman}, {\em Fixed points of involutive automorphisms of the Bruhat order}, Adv. Math., {\bf 195}(1) (2005), 283-¨C296.

\bibitem{Hu2}
\leavevmode\vrule height 2pt depth -1.6pt width 23pt,
{\em The Combinatorics of twisted involutions in Coxeter groups}, Trans. Amer. Math. Soc., {\bf 359}(6) (2007), 2787--2798.

\bibitem{Hum} {\sc J.E.~Humphreys}, {\em Reflection Groups and Coxeter Groups},
  Cambridge Studies in Advanced Mathematics, Vol. {\bf 29}, Cambridge Univ. Press,
  Cambridge, UK, 1990.

\bibitem{HZH1}
{\sc J.~Hu, J. Zhang}, {\em On involutions in symmetric groups and a conjecture of Lusztig}, Adv. Math., {\bf 364} (2016), 1189¡ª-1254.

\bibitem{Ko}
{\sc R.E.~Kottwitz}, {\em Involutions in Weyl groups}, Representation Theory, {\bf 4} (2000), 1--15.

\bibitem{Lu1} {\sc G.~Lusztig}, {\em A bar operator for involutions in a Coxeter groups}, (2012), preprint, arXiv:1112.0969.

\bibitem{Lu2}
\leavevmode\vrule height 2pt depth -1.6pt width 23pt, {\em Asymptotic Hecke algebras and involutions}, (2012), preprint, arXiv:1204.0276.

\bibitem{Lu3}
\leavevmode\vrule height 2pt depth -1.6pt width 23pt, {\em  An involution based left ideal in the Hecke algebra}, Representation Theory, {\bf 20}, (8), (2016), 172--186.

\bibitem{LV1}
{\sc G.~Lusztig, D.~Vogan}, {\em Hecke algebras and involutions in Weyl groups}, Bulletin of the Institute of Mathematics Academia Sinica (New Series), {\bf 7}(3) (2012), 323--354.

\bibitem{Mar}
{\sc E.~Marberg}, {\em Positivity conjectures for Kazhdan-Lusztig theory on twisted involutions: the universal case}, Represent. Theory, {\bf 18}  (2014), 88--116.

\bibitem{Mat} {\sc H.~Matsumoto}, {\em G\'{e}n\'{e}rateurs et relations des groupes de Weyl g\'{e}n\'{e}ralis\'{e}s}, C. R. Acad. Sci.
Paris, {\bf 258}, (1964), 3419--3422.

\bibitem{RS1}
{\sc R.W.~Richardson, T.A.~Springer}, {\em The Bruhat order on symmetric varieties}, Geom. Dedicata, {\bf 35} (1990), 389--436.

\bibitem{RS2}
\leavevmode\vrule height 2pt depth -1.6pt width 23pt, {\em Complements to: ``The Bruhat order on symmetric varieties"}, Geom. Dedicata, {\bf 49} (1994), 231--238.

\bibitem{WY}
{\sc B.J.~Wyser, A.~Yong}, {\em Polynomials for symmetric orbit closures in the flag variety}, Transform. Groups, to appear.

\end{thebibliography}
\end{document}